%% file: TH20251119.tex
\documentclass[ hidelinks, onefignum, onetabnum]{siamart220329}
\makeatletter
\def\refstepcounter@optarg[#1]#2{%
  \cref@old@refstepcounter{#2}%
  \cref@constructprefix{#2}{\cref@result}%
  \@ifundefined{cref@#1@alias}%
    {\def\@tempa{#1}}%
    {\def\@tempa{\csname cref@#1@alias\endcsname}}%
  \protected@edef\cref@currentlabel{%
    [\@tempa][\arabic{#2}][\cref@result]%
    \csname p@#2\endcsname\csname the#2\endcsname}}%
\makeatother
\usepackage{amsmath,amsfonts}
\usepackage{graphicx}
\usepackage{float}
\usepackage{tikz}
\usetikzlibrary{decorations.fractals,spy}
\usepackage{algorithmic}
\usepackage{mathrsfs}
\usepackage[shortlabels]{enumitem} 
\usepackage{algorithm}

\usepackage{multicol,multirow}
\usepackage{booktabs}
\usepackage{adjustbox}
\usepackage{pgfplots}
\usepgfplotslibrary{groupplots}
\pgfplotsset{compat=1.18}

\usepackage{xcolor}

\input{symboldef.tex}

\newcommand{\rr}{\mathbb{R}}

\title{
Truncated Huber Penalty for Sparse Signal Recovery with Convergence Analysis \thanks{This work is supported in part by the National Natural Science Foundation of China (Grant No. 12361089), the National Group for Scientific Computation (GNCS-INDAM), Research Projects 2025.}
}

\author{Li Yang\thanks{School of Mathematics and Statistics, Hunan Normal University, Changsha, Hunan, China
 ({li.yang5@unibo.it)}.}
\and Serena Morigi \thanks{Department of Mathematics, University of Bologna, 
40127 Bologna, Italy  ({serena.morigi@unibo.it})}
\and Michael K. Ng\thanks{Department of Mathematics, Hong Kong Baptist University, Kowloon Tong, Hong Kong,  China  ({michael-ng@hkbu.edu.hk}).}
\and You-wei Wen\thanks{Corresponding author. Key Laboratory of Computing and Stochastic Mathematics (LCSM), School of Mathematics and Statistics, Hunan Normal University, Changsha, Hunan, China 
 ({wenyouwei@gmail.com})} 
}

\begin{document}

\maketitle

\begin{abstract}
Sparse signal recovery from under-determined systems presents significant challenges when using conventional $L_0$ and $L_1$ penalties, primarily due to computational complexity and estimation bias. This paper introduces a truncated Huber penalty, a non-convex metric that effectively bridges the gap between unbiased sparse recovery and differentiable optimization. The proposed penalty applies quadratic regularization to small entries while truncating large magnitudes, avoiding non-differentiable points at {optimal solutions}.
Theoretical analysis demonstrates that, for an appropriately chosen threshold, any {$(s,\mu)$}-sparse solution recoverable via conventional penalties remains a local optimum under the truncated Huber function. This property allows the exact and robust recovery theories developed for other penalty regularization functions to be directly extended to the truncated Huber function. 
To solve the optimization problem, we develop a block coordinate descent (BCD) algorithm with \emph{finite-step convergence guarantees} under spark conditions.
{Numerical experiments validate the effectiveness and robustness of the proposed method in both synthetic and real scenarios.
Furthermore, we demonstrate the flexibility of the truncated Huber framework through two extensions: one to an adaptively weighted variant inspired by sorted penalties, and another to the gradient domain for applications such as signal denoising and image smoothing. }
\end{abstract}

\begin{keywords}
truncated Huber penalty, sparse signal recovery, block coordinate descent, signal denoising, image smoothing, non-convex optimization   
\end{keywords}

\begin{MSCcodes}
90C26, 90C90, 65K10,  49N45, 
\end{MSCcodes}

\section{Introduction}
Signal recovery is a fundamental problem in signal processing, with applications in signal reconstruction \cite{1580791,tan2018digital,li2023total}, image processing \cite{rudin1992nonlinear}, compressed sensing \cite{rath2008}, and machine learning \cite{10.1111/j.2517-6161.1996.tb02080.x}.  
This task is mathematically formulated as recovering an unknown signal $\vx \in \mathbb{R}^{n}$ from a set of linear measurements. In the noise-free case, the relationship between the signal and the observations is given by:
\[\vb = A\vx,\] 
where $A \in \mathbb{R}^{m\times n}$ is a sensing matrix with full row rank, serving as a transformation or projection operator from the original signal space to the measurement space. 
In practical scenarios, observations are typically corrupted by noise, leading to the model:
\[\vb = A\vx+\vn,\]
where $\vn\in\rr^m$ represents noise with $\|\vn\|_2^2 \leq \epsilon^2$ for a given noise level $\epsilon$. 
We assume $\vb \neq \mathbf{0}$ throughout the paper.

The objective of signal recovery is to find the best estimate $\hat{\vx}$ of the unknown signal $\vx$ by minimizing a penalty term $\Phi:\rr^n\rightarrow\rr$ subject to the constraint:
\begin{equation}\label{constrain_modela}
\min_{\vx \in \mathcal{S}} \Phi(\vx),
\end{equation}
where the feasible set $\mathcal{S}$ is defined as $\mathcal{S}=\{\vx \mid A\vx = \vb\}$ for the noise-free case, or $\mathcal{S}= \{\vx \mid \|A\vx - \vb\|_2^2 \leq \epsilon^2\}$ for the noisy case.

Over the years, a wide variety of penalty functions have been proposed and studied in the literature \cite{doi:10.1137/S0097539792240406, doi:10.1137/S003614450037906X, 4303060, lv2009unified, yin2014ratio}, each depends on difference signal characteristics and prior knowledge.
In many applications, signals are often assumed to be sparse or compressible, meaning that only a few entries of the signal are significantly non-zero, while the majority are either zero or close to zero. 
This sparsity assumption enables signal representation using a reduced number of parameters, allowing for recovery from relatively few measurements, thereby reducing the sampling and storage requirements.

The classical $L_{0}$ norm penalty \cite{doi:10.1137/S0097539792240406}, which directly quantifies sparsity by counting the number of non-zero entries in a vector (i.e., $\Phi(\vx)=\norm{\vx}_0$), unfortunately poses a NP-hard optimization problem due to its discrete and discontinuous property. 
As an alternative, the basis pursuit problem (BP), which adopts the $L_{1}$  norm penalty \cite{doi:10.1137/S003614450037906X}, has been the subject of extensive research. 
The $L_{1}$ norm serves as the convex envelope of the $L_{0}$ norm and is computationally effective in inducing sparsity among convex penalties. 
Nevertheless, the $L_{1}$ norm penalty {exhibits a limitation} in that it tends to underestimate high-amplitude entries and fails to distinguish between large and small non-zero entries.
Subsequently, non-convex penalties, such as {the capped $L_{1}$ penalty ($CL_1$) \cite{pmlr-v28-gong13a},}  smoothly clipped absolute deviation ($SCAD$) \cite{fan2001variable}, minimax concave penalty ($MCP$) \cite{zhang2010nearly}, the $L_{p}$ penalty ($0 < p < 1$)  \cite{4303060, lai2013improved}, the transformed $L_{1}$ penalty ($TL_1$)  \cite{lv2009unified,zhang2018minimization,huo2022stable},  the $L_{1}\mbox{-}L_2$ penalty \cite{yin2015minimization,ge2021new,ma2017truncated,yan2017sparse}, the $L_1/L_2$ penalty \cite{hoyer2002non, wang2021limited,tao2022minimization,wang2024sorted,tao2024partly} and so on, have emerged as areas of interest. 
On the computational side, non-convex penalty are generally more challenging to minimize. 
Algorithms for directly minimizing  these penalties include compressive sampling matching pursuit (CoSaMP) \cite{do2008sparsity} and iterative hard thresholding (IHT) \cite{needell2009cosamp,he2024relu} for $L_0$ minimization, iteratively reweighted least squares (IRLS)  \cite{daubechies2010iteratively} for $L_1$ penalized problems,  
convex non-convex (CNC) \cite{lanza2019sparsity} for $MCP$ penalized problems,
{alternating direction method of multipliers (ADMM) \cite{tao2022minimization} for $L_1/L_2$ penalty model,} and difference of convex functions algorithm (DCA) \cite{lou2015computational}  for $L_{1 }\mbox{-}L_2$ minimization and so on.

Conventional sparsity-promoting penalties quantify sparsity solely by counting non-zero entries, which proves inadequate for real-world compressible signals. Such signals typically contain a few dominant coefficients and numerous negligible entries that approximate zero. 
Traditional norms limit sparse representation by assigning equal weights to entries regardless of magnitude, thus failing to distinguish between trivial and significant non-zero values \cite{eamaz2022,5238742}. 
The $L_{0,\epsilon}$ norm addresses this issue by counting only entries exceeding a predefined threshold  $\epsilon$ \cite{eamaz2022,rath2008}, yet it risks disregarding small but meaningful coefficients.
To overcome these limitations, we introduce the truncated Huber ($TH$) penalty, a variant of the classical Huber function \cite{huber1992robust}  originally developed for robust regression. 
Unlike the classical Huber penalty, whose constant derivative for large magnitudes introduces estimation bias, 
the truncated formulation mitigates this issue. 
The $TH$ function $\phi_{\mu}$ with parameter $\mu\in\rr_{+}$ is defined by
\begin{equation}\label{truncatedHF}
\phi_{\mu}(x):=\min\left(1, \tfrac{x^2}{\mu^2 }\right)=\left\{
\begin{aligned}
&\tfrac{1}{\mu^2 }x^2, & |x|\leq\mu ,\\
&1,&|x|> \mu ,
\end{aligned}
\right.
\end{equation}
where the parameter $\mu$ serves as a predefined threshold, effectively differentiating negligible non-zero entries from significant ones in the signal. 
This definition enables a quadratic penalty for entries with magnitudes below $\mu$, while imposing a constant penalty on larger magnitudes. 
Consequently, the function is sensitive to the relative importance of signal components without uniformly penalizing all non-zero elements, addressing a key limitation of conventional sparsity-inducing penalties.
In Figure \ref{diffregularfct}, we plot several popular scalar penalties together with the proposed $TH$ penalty for different $\mu$ values, including the $L_0$ penalty, the $L_1$ penalty, {the $CL_1$ penalty with $\mu=1,$} the $TL_1$ penalty with $\beta=1$, the $MCP$ penalty with $\mu=0.75$, and the classical Huber penalty with $\mu=0.5$. 
For small $\mu$, the $TH$ penalty approximates the $L_0$ norm while maintaining continuity, overcoming the reliance on pure non-zero element counts for sparsity quantification. 
Furthermore, unlike the $L_1$ norm, which can introduce bias in estimating large coefficients by penalizing all non-zero entries based on their amplitudes, the $TH$ penalty avoids this issue by imposing a constant penalty only on magnitudes exceeding $\mu$. 
The $TH$ penalty has demonstrated effectiveness in various applications, including image fusion \cite{jie2024multi} and image smoothing \cite{9490302}.

\begin{figure}
    \centering
    \includegraphics[width=1.0\linewidth]{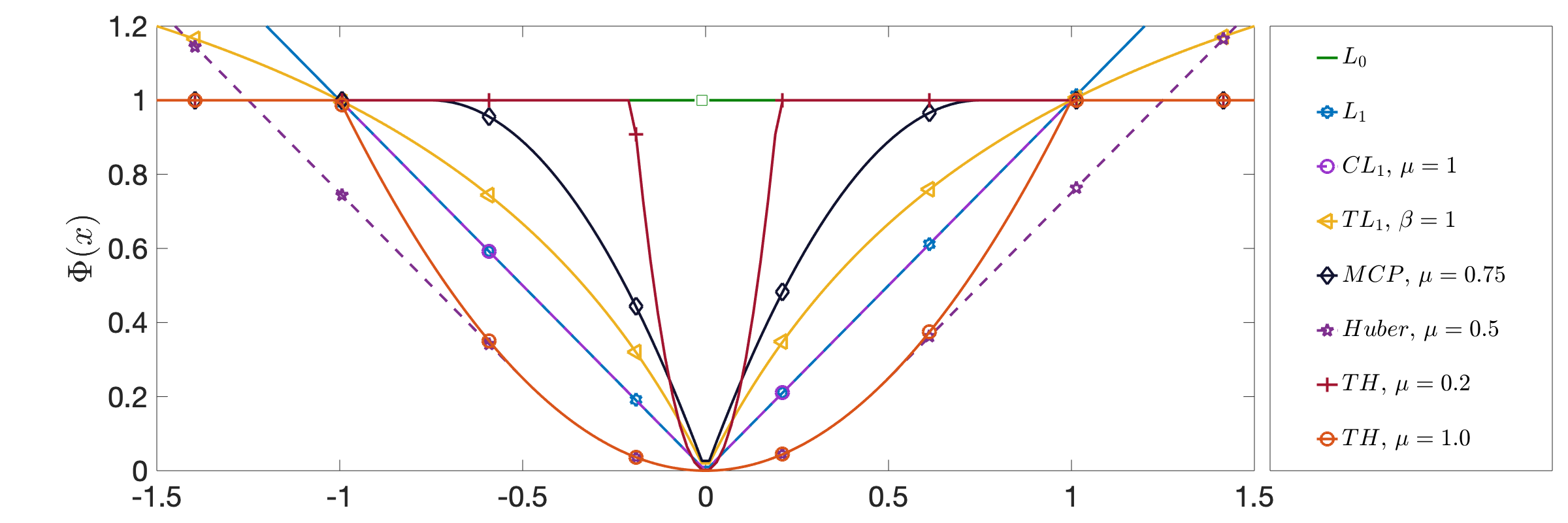}
    \caption{{Scalar penalties for $L_0(|x|^0)$, $L_1(|x|)$, $CL_1 (\min(|x|,\mu))$, $TL_1 (\tfrac{(\beta+1)|x|}{\beta + |x|})$, $MCP (|x| - \tfrac{x^2}{2\mu}$ if $ |x| \leq \mu$ otherwise $\tfrac{\mu}{2})$, $Huber(\tfrac{1}{2\mu}x^2$ if $|x|<\mu$ otherwise $|x|-\tfrac{\mu}{2})$, and $TH$.}}
    \label{diffregularfct}
\end{figure}

Using the $TH$ function \eqref{truncatedHF}, we propose the following model for noise-free signal recovery:
\begin{equation} \label{constrain_model}
\min_{\vx \in \cS} \Phi_{\mu}(\vx), 
\end{equation}  
where $\cS=\{\vx\in\rr^n\mid A\vx=\vb\}$, and $\Phi_{\mu}(\vx)=\sum_{i = 1}^n\phi_{\mu}(\vx_i)$.
In the presence of noise, we consider the following unconstrained optimization problem involving a cost function $\cG_{\mu}:\rr^n\rightarrow\rr$,
\begin{equation}\label{obj_singleu}
\min_{\vx} \mathcal{G}_{\mu}(\vx) := \tfrac{\alpha}{2} \norm{A\vx - \vb}_2^2 + \Phi_{\mu}(\vx),
\end{equation}
where $\alpha > 0$ is a regularization parameter that balances measurement fidelity and regularization.

The $TH$ penalty has several advantages. 
As $\mu$ approaches zero, the function converges to the $L_0$ norm while retaining continuity, facilitating analytical tractability.
Unlike the $L_1$ norm, it avoids bias in estimating large coefficients. 
{Moreover, we show that {any $(s,\mu)$-sparse} solution recoverable via conventional penalties remains a local optimum under the $TH$ formulation (see Theorem \ref{TH:s-sparse}).}
The $TH$ penalty is differentiable everywhere except at $|x|=\mu$,
but we prove that the optimal solution of both  \eqref{constrain_model} and \eqref{obj_singleu} never attains this value,  ensuring differentiability at the optimum, see Theorems \ref{THdifferentiable_refined} and \ref{Gfuncontinus}. 
From a computational perspective, we develop a block coordinate descent (BCD \cite{tseng2001convergence}) algorithm using a surrogate function.  We further establish theoretical guarantees for finite-step convergence under a spark condition, see Theorems \ref{covgergalg1} and \ref{convergalg2}.
Additionally, we propose a $\mu$-continuation strategy, which initializes  $\mu$ at a large value to ensure matrix invertibility and progressively reducing it to stabilize optimization and mitigate numerical instability.

The rest of this paper is organized as follows. 
In section \ref{introhuberfunct}, we derive optimality conditions for the truncated Huber-penalized models \eqref{constrain_model} and \eqref{obj_singleu}.
We further show that under the $TH$ framework, sparse feasible vectors are locally optimal solutions.
Section \ref{BCDmethod} presents the BCD algorithm for minimizing the proposed models. 
Section \ref{sec_experiments} provides numerical results validating the effectiveness of the proposed methods. 
Finally, conclusions and future works are given in section \ref{sec:concl}.

\section{Truncated Huber-penalized model for signal recovery}\label{introhuberfunct}
In this section, we first recall some definitions which will be useful for the rest of the work. 

{For a finite set \(\mathcal{S}\), the {cardinality} of \(\mathcal{S}\), denoted \(\text{card}(\mathcal{S})\), is the number of elements in \(\mathcal{S}\). 
}
The support of a vector $\vx\in\rr^n$ is the index set of its non-zero entries:
\[
\mathrm{supp}(\vx):=\{i:\vx_i\neq 0\}. 
\]
A vector $\vx$ is called $s$-sparse if it contains at most $s$ non-zero entries, i.e., $\norm{\vx}_0=\mathrm{card}(\mathrm{supp}(\vx))\leq s$ with $s<n$. 
{A vector $\vx$ is called $(s,\mu)$-sparse if it contains at most $s$ non-zero entries whose magnitudes are larger than $\mu$, i.e., $\norm{\vx}_0 \leq s< n$ and $ |\vx_i| > \mu$ holds for all $i\in\mathrm{supp}(\vx)$.}
The spark of a matrix $A$ is the smallest number of linearly dependent columns:
\[
\mathrm{spark}(A)=\min_{\vv\neq\mathbf{0}}\norm{\vv}_0,\quad \text{s.t.}\quad A\vv=\mathbf{0}.
\]
The $TH$ function, defined in \eqref{truncatedHF},
is continuous and satisfies
\[
\lim_{x\to\pm\mu }\phi_{\mu}(x)=\phi_{\mu}(\pm\mu), \quad \text{and}\quad
\lim_{\mu \to 0} \phi_\mu(x) = \begin{cases}  
0, & x = 0, \\  
1, & x \neq 0.  
\end{cases}  
\]

For any vector $\vx\in\rr^n$, the function $\Phi_{\mu}(\vx)$ is additively separable with respect to the components of $\vx$.
Furthermore, for any $\vx\in \rr^n$ and $\mu>0$,  it satisfies 
\[
0\leq\Phi_{\mu}(\vx)\leq n,\quad \text{and}\quad \lim_{\mu\rightarrow 0}\Phi_{\mu}(\vx)=\|\vx\|_0.
\]
Thus, $\Phi_{\mu}(\vx)$ is bounded and provides a continuous, unbiased approximation of the $L_0$ norm.
However, $\Phi_\mu(\vx)$ remains non-convex and non-differentiable at points where any component of $\vx$ equals $\pm\mu$.
\begin{lemma}\label{GradPhiTH}
The function $\Phi_{\mu}(\vx)$ is differentiable at any $\vx \in \mathbb{R}^n$ with $|\vx_i| \neq \mu$ for all  $i = 1, \ldots, n.$ 
Specifically, for such $\vx$,
the gradient of $\Phi_{\mu}(\vx)$ is given by 
\[
\nabla\Phi_{\mu}(\vx) = \frac{2}{\mu^2}(\mathbf{1} - \vomega) \circ \vx,
\]
where $\circ$ represents the Hadamard product, and $\vomega=\cH_{\mu}(\vx)\in\mathbb{R}^n$, with $\cH_{\mu}(\vx)$ a `filter' operator whose $i$-th entry is defined component-wise by
\begin{equation}\label{porpoptw}
   \big( \cH_{\mu}(\vx)\big)_i =\begin{cases}  
0, & \text{if } |\vx_i| \leq \mu,\\
1, & \text{if } |\vx_i| > \mu,
\end{cases}
\quad \forall i \in \{1, \ldots, n\}.
\end{equation}
\end{lemma}

\subsection{Proximal mapping of $\Phi_\mu$} 
We analyze the proximal mapping of the $TH$ penalty $\Phi_\mu$. 
Since $\Phi_\mu$ is a non-convex vector variable function separable across variables, we focus on an efficient way that reduces the proximal map of $\Phi_\mu$ to the proximal maps of its individual components $\phi_\mu$:
\begin{equation}
\operatorname{Prox}_{\lambda}^{\phi_{\mu}}(\va_i) = \arg\min_{\vx_i\in\rr}\left\{\phi_{\mu}(\vx_i)+\tfrac{1}{2\lambda}(\vx_i - \va_i)^2\right\},\quad i=1,\ldots,n,
\end{equation}
for every $\lambda>0$.
The following result characterizes the optimal solution to each subproblem.
\begin{theorem} \label{decompprosolut}
Let $\psi(x;a) = \phi_{\mu}(x)+\frac{1}{2\lambda}(x - a)^2$, where $\phi_{\mu}(x)$ is defined in \eqref{truncatedHF}. The optimal solution of $\min_{x} \psi(x;a)$ is:
\begin{equation}\label{solutionforproxi}
  x^\dagger = \begin{cases}
a, & \text{if } |a| > \sqrt{\mu^2 + 2\lambda}, \\
a \lor \frac{\mu^2}{\mu^2 + 2\lambda}a, & \text{if } |a| = \sqrt{\mu^2 + 2\lambda}, \\
\frac{\mu^2}{\mu^2 + 2\lambda}a, & \text{if } |a| < \sqrt{\mu^2 + 2\lambda}.
\end{cases}  
\end{equation}
{Here $\lor$ denotes the logical “or”.}
Moreover, $x^\dagger$ cannot attain the values $\pm\mu$.
\end{theorem}
\begin{proof}
The derivative of $\phi_{\mu}(x)$ is $\phi_{\mu}'(x) = 0$ for $|x| > \mu$ and $\phi_{\mu}'(x)=\frac{2x}{\mu^2}$ for $|x| < \mu$. The first-order optimality condition  $\psi'(x^\dagger;a) = 0$ results in
\[
x^\dagger - a = 0, \quad \text{or} \quad \tfrac{2\lambda}{\mu^2}x^\dagger + x^\dagger - a = 0.
\]
Solving these equations yields $x^\dagger = a$ or $x^\dagger=\frac{\mu^2}{\mu^2 + 2\lambda}a$.

Next, we evaluate $\psi(x;a)$ at these potential solutions:
\[
\psi(a;a) = 1,  \quad \text{and} \quad \psi\left(\tfrac{\mu^2}{\mu^2 + 2\lambda}a;a\right)=\tfrac{a^2}{\mu^2 + 2\lambda}.
\]
Additionally, we consider the values at the non-differentiable points, i.e., $\psi(\pm\mu;a)=1+\tfrac{(\pm\mu-a)^2}{2\lambda}$. 
Subsequently, we analyze the value of $a$ in the following cases:
\begin{enumerate}[(1)]
\item If $|a| = \sqrt{\mu^2 + 2\lambda}$, then $\psi(a;a)=\psi\left(\frac{\mu^2}{\mu^2 + 2\lambda}a;a\right)$, so both choices are optimal.
\item If $|a| < \sqrt{\mu^2 + 2\lambda}$, then $\psi(\pm\mu;a)>\psi(a;a)>\psi\left(\frac{\mu^2}{\mu^2 + 2\lambda}a;a\right)$. In this case, $x^\dagger=\frac{\mu^2}{\mu^2 + 2\lambda}a$ minimizes $\psi(x;a)$, and $|x^\dagger| < \mu$.
\item If $|a| > \sqrt{\mu^2 + 2\lambda}$, then $\psi(a;a)<\psi\left(\frac{\mu^2}{\mu^2 + 2\lambda}a;a\right)$ and $\psi(a;a)<\psi(\pm\mu;a)$, making $x^\dagger = a$ the optimal solution, with $|x^\dagger| > \mu$.
\end{enumerate}
Therefore, condition \eqref{solutionforproxi} holds. The optimal solution $x^\dagger$ remains equal to $a$ when $|a|>\sqrt{\mu^2 + 2\lambda}$. For $|a|<\sqrt{\mu^2 + 2\lambda}$, $x^\dagger$ is scaled by $\tfrac{\mu^2}{(\mu^2 + 2\lambda)}$, ensuring that $|x^\dagger|=\tfrac{\mu^2}{(\mu^2 + 2\lambda)}|a|$ strictly less than $\mu$. 
These imply that $x^\dagger=\arg\min_{x}\psi(x;a)$ cannot attain the values $\pm\mu$. 
\end{proof}

{For closed-form proximal mappings of the \(L_1\), \(TL_1\), MCP and \(CL_1\) penalties, we refer readers to \cite{AN2022319, zhang2017minimization, pmlr-v28-gong13a}.}

\subsection{Signal recovery models}
We now analyze signal recovery models using the $TH$ penalty $\Phi_{\mu}(\vx)$ in both noise-free and noisy settings, formulated as \eqref{constrain_model} and \eqref{obj_singleu}, respectively.
We consider introducing the unified optimization formulation \eqref{constrain_modela}, with $\Phi(\vx)=\Phi_{\mu}(\vx)$.

Unlike the $L_1/L_2$ model, which requires the strong null space property (sNSP)  to ensure local optimality of sparse vectors \cite{rahimi2019scale}, our model guarantees this property without additional structural assumptions on $A$, as shown in the following result.

\begin{theorem} \label{TH:s-sparse} 
{
Let $\wvx$ be a $(s,\mu)$-sparse vector satisfying $A\wvx = \vb$ (or $\|A\wvx - \vb\|_2^2 \leq \epsilon^2$).} Then $\wvx$ is a local minimizer of  $\min_{\vx\in\cS} \Phi_{\mu}(\vx)$. Specifically, there exists $\rho > 0$ such that for any perturbation $\vv$ with $\|\vv\|_\infty < \rho$ (or $\|\vv\|_\infty < \min(\rho, \tfrac{\epsilon}{\sqrt{n}\|A\|_2})$ in the noisy case) and $(\wvx + \vv) \in \cS$,  
\[  
\Phi_{\mu}(\wvx + \vv) \geq \Phi_{\mu}(\wvx).  
\]  

\end{theorem}  
\begin{proof}
Define $\rho:=\tfrac{1}{2}\min_{i}\{\abs{\abs{\wvx_i}-\mu}\}$, ensuring  $0<\rho\leq \mu/2$. 
For any $\vv$ satisfying $0<\norm{\vv}_{\infty}< \rho$ and $A\vv=\mathbf0$ (or  $\vv\in  \{\norm{A(\wvx+\vv)-\vb}_2^2\leq \epsilon^2\}$, $0<\norm{\vv}_\infty<\min(\rho,\tfrac{\epsilon}{\sqrt{n}\norm{A}_2})$), the following properties hold:
\begin{enumerate}[(1)]
    \item If $i\in \mathrm{supp}(\wvx)$, then $\abs{\wvx_i}>\mu$ and
 $\abs{\wvx_i+\vv_i}>\mu $. Hence,  $\phi_{\mu}(\wvx_i+\vv_i)=\phi_{\mu}(\wvx_i)$.
    \item If $i\notin \mathrm{supp}(\wvx)$, then $\abs{\wvx_i}=0<\mu$ and $\abs{\vv_i}<\mu $. Thus, $\phi_{\mu}(\wvx_i+\vv_i)=\tfrac{\vv_i^2}{\mu^2}$.
\end{enumerate}
Summing over $i$,  we have
\[
\Phi_{\mu}(\wvx)=\sum_{i\in \mathrm{supp}(\wvx)}\phi_{\mu}(\wvx_i)+\sum_{i\notin \mathrm{supp}(\wvx)}\phi_{\mu}(\wvx_i)
=\sum_{i\in \mathrm{supp}(\wvx)}\phi_{\mu}(\wvx_i+\vv_i).
\]
Notice that $\Phi_{\mu}(\wvx+\vv)=\sum_{i\in \mathrm{supp}(\wvx)}\phi_{\mu}(\wvx_i+\vv_i)+\sum_{i\notin \mathrm{supp}(\wvx)}\phi_{\mu}(\wvx_i+\vv_i)$ with  the second term being nonnegative, the lemma holds. 
\end{proof}

The key aspect established by Theorem \ref{TH:s-sparse} ensures that the non-zero components of $\wvx$ are preserved under the optimization framework. On the other hand,  the parameter $\mu$ plays a strategic role as a threshold, which allows to distinguish between insignificant and significant non-zero entries.
An appropriate choice of $\mu$ is crucial for maintaining small yet meaningful coefficients while promoting sparsity. 
To characterize the relationship between $\mu$ and the optimal value, we present the following lemma, which quantifies this trade-off by showing that the optimal objective value $\Phi_{\mu}(\widehat{\vx}_{\mu})$ non-increasing as $\mu$ increases.   

\begin{lemma}\label{mudecre_objincre}  
Let $\widehat{\vx}_{\mu} = \arg\min_{\vx \in \cS} \Phi_{\mu}(\vx)$. For $\mu_1 > \mu_2 > 0$, 
\[  
\Phi_{\mu_1}(\widehat{\vx}_{\mu_1}) \leq \Phi_{\mu_2}(\widehat{\vx}_{\mu_2}).  
\]  
\end{lemma}  

\begin{proof}  
For $\mu_1 > \mu_2$, we have $\min(1, x^2/\mu_2^2) \geq \min(1, x^2/\mu_1^2)$, implying $\Phi_{\mu_1}(\vx) \leq \Phi_{\mu_2}(\vx)$ for all $\vx \in \mathbb{R}^n$. Consequently, 
\[ 
\min_{\vx\in\cS} \Phi_{\mu_1}(\vx)=\Phi_{\mu_1}(\wvx_{\mu_1})\leq \Phi_{\mu_1}(\wvx_{\mu_2}) \leq \Phi_{\mu_2}(\wvx_{\mu_2})=\min_{\vx\in\cS} \Phi_{\mu_2}(\vx). 
\]
Hence the lemma holds.
\end{proof}  

\subsection{Differentiability at optimal solutions}  

The $TH$ penalty $\Phi_{\mu}(\vx)$ is non-differentiable at points where any component equals $\pm\mu$, which poses challenges for numerical optimization.  
To overcome this issue, we demonstrate that the minimizes of both the noise-free case (see \eqref{constrain_model}) and the noisy case (see \eqref{obj_singleu}) cannot attain these non-differentiable points. 
Consequently, $\Phi_{\mu}(\vx)$ remains differentiable in a neighborhood around minimizer, thereby facilitating efficient optimization.  

\begin{theorem}\label{THdifferentiable_refined}  
Let $A \in \mathbb{R}^{m \times n}$ with $m < n$. For the noise-free case \eqref{constrain_model}, we assume that any $(n-1)$-column submatrix of $A$ has full row rank.
Then, the optimal solution $\wvx\in\rr^n $ of \eqref{constrain_model} satisfies  
\[
|\widehat\vx_i| \neq \mu, \quad \forall i = 1, \ldots, n,  
\]  
and $\Phi_{\mu}(\vx)$ is differentiable in a neighborhood of $\wvx$.  
\end{theorem}  

\begin{proof}
Assume, by contradiction, that $|\wvx_i| = \mu$ for some $i$. Define index sets $\cI_0 = \{i \mid |\wvx_i| = \mu\}$, $\cI_1 = \{i \mid |\wvx_i| > \mu\}$, and $\cI_2 = \{i \mid |\wvx_i| < \mu\}$. Since $\cI_0 \neq \emptyset$, construct $\vv \in \mathbb{R}^n$ such that $A\vv = \mathbf{0}$, $\|\vv\|_2 = 1$, and $\vv_i \neq 0$ for $i \in \cI_0$.  

Define $\rho := \tfrac{1}{2} \min_{i \in \cI_1 \cup \cI_2} \{||\wvx_i| - \mu|\}$. For any $0 < \epsilon < \rho$, the perturbation $\wvx \pm \epsilon \vv$ remains feasible. Partition $\cI_0$ into $\cI_{01} = \{i \mid \wvx_i \vv_i \geq 0\}$ and $\cI_{02} = \{i \mid \wvx_i \vv_i < 0\}$.
This implies that for $i \in \cI_{01}$, $\wvx_i$ and $\vv_i$ have the same sign, while for $i \in \cI_{02}$, $\wvx_i$ and $\vv_i$ have opposite signs. 
Hence,
\begin{equation}\label{muidxsumterm}
\sum_{i\in\cI_{01}}\wvx_i\vv_i-\sum_{i\in\cI_{02}}\wvx_i\vv_i=\mu\sum_{i\in\cI_{0}}|\vv_i|>0.
\end{equation}

We now analyze the behavior of $\phi_{\mu}(\wvx_i\pm\epsilon\vv_i)$ for sets $\cI_{01}$, $\cI_{02}$, $\cI_1$, and $\cI_2$:
\begin{enumerate}[1),leftmargin=*]
    \item  For $i\in \cI_1$, we have $\abs{\wvx_i\pm \epsilon\vv_i}>\mu $, and 
    $
    \phi_{\mu}(\wvx_i\pm\epsilon\vv_i)=\phi_{\mu}(\wvx_i). 
    $
     \item  For $i\in \cI_2$, we obtain $\abs{\wvx_i\pm\epsilon\vv_i}<\mu $, and
     $
    \phi_{\mu}(\wvx_i\pm \epsilon\vv_i)=\tfrac{1}{\mu^2}(\wvx_i \pm \epsilon\vv_i)^2=\phi_{\mu}(\wvx_i) + \tfrac{1}{\mu^2}\left(\epsilon^2\vv_i^2\pm 2\epsilon\wvx_i \vv_i\right). 
    $
    \item  For $i\in \cI_{01}$, we can infer that $\abs{\wvx_i+\epsilon\vv_i}>\mu, \abs{\wvx_i-\epsilon\vv_i}<\mu $. Then
    $
    \phi_{\mu}(\wvx_i+\epsilon\vv_i)=\phi_{\mu}(\wvx_i)$  and $\phi_{\mu}(\wvx_i-\epsilon\vv_i)=\phi_{\mu}(\wvx_i) + \tfrac{1}{\mu^2}\left(\epsilon^2\vv_i^2- 2\epsilon\wvx_i \vv_i\right).
    $
    \item  For $i\in \cI_{02}$, we have $\abs{\wvx_i+\epsilon\vv_i}<\mu, \abs{\wvx_i-\epsilon\vv_i}>\mu $. 
    Then 
    $    \phi_{\mu}(\wvx_i-\epsilon\vv_i)=\phi_{\mu}(\wvx_i)$  and  $\phi_{\mu}(\wvx_i+\epsilon\vv_i)=\phi_{\mu}(\wvx_i) + \tfrac{1}{\mu^2}\left(\epsilon^2\vv_i^2+2 \epsilon\wvx_i \vv_i\right).
    $
\end{enumerate}

Next, we analyze the sign of $\sum_{i \in \cI_{01}} \wvx_i \vv_i + \sum_{i \in \cI_2} \wvx_i \vv_i$ to find a descent direction near $\wvx$ for $\Phi_{\mu}(\wvx)$. 
If $\sum_{i \in \cI_{01}} \wvx_i \vv_i + \sum_{i \in \cI_2} \wvx_i \vv_i > 0$,  expanding the expression  $\Phi_{\mu}(\wvx - \epsilon\vv) = \sum_{i}\phi_{\mu}(\wvx_i - \epsilon\vv_i)$  gives
\[
    \Phi_{\mu}(\wvx - \epsilon \vv) = \Phi_{\mu}(\wvx) + \tfrac{1}{\mu^2}\big(\epsilon^2 \sum_{i \in \cI_{01} \cup \cI_2} \vv_i^2 - 2\epsilon (\sum_{i \in \cI_{01}} \wvx_i \vv_i + \sum_{i \in \cI_2} \wvx_i \vv_i)\big).
\]
Choosing $\epsilon < \min \big\{2 \tfrac{\sum_{i \in \cI_{01}} \wvx_i \vv_i + \sum_{i \in \cI_2} \wvx_i \vv_i}{\sum_{i \in \cI_{01} \cup \cI_2} \vv_i^2}, \rho\big\}$, we obtain $\Phi_{\mu}(\wvx - \epsilon \vv) < \Phi_{\mu}(\wvx)$.

If $\sum_{i\in\cI_{01}}\wvx_i\vv_i+\sum_{i\in\cI_2}\wvx_i\vv_i \leq 0$,
then $\sum_{i\in\cI_2}\wvx_i\vv_i \leq -\sum_{i\in\cI_{01}}\wvx_i\vv_i$.
Expanding $\Phi_{\mu}(\wvx + \epsilon\vv) = \sum_{i}\phi_{\mu}(\wvx_i + \epsilon\vv_i)$ yields
\begin{align*}
\Phi_{\mu}(\wvx + \epsilon\vv) &= \Phi_{\mu}(\wvx) + \tfrac{1}{\mu^2}\big(\epsilon^2\sum_{i\in\cI_{02}\cup\cI_2}\vv_i^2 + 2\epsilon\sum_{i\in\cI_{02}}\wvx_i\vv_i + 2\epsilon\sum_{i\in\cI_2}\wvx_i\vv_i\big) \\
&\leq \Phi_{\mu}(\wvx) + \tfrac{1}{\mu^2}\big(\epsilon^2\sum_{i\in\cI_{02}\cup\cI_2}\vv_i^2 - 2\epsilon(\sum_{i\in\cI_{01}}\wvx_i\vv_i - \sum_{i\in\cI_{02}}\wvx_i\vv_i)\big). 
\end{align*}
Combining this with \eqref{muidxsumterm}, when we choose $\epsilon < \min\big\{\tfrac{2\mu\sum_{i\in\cI_{0}}|\vv_i|}{\sum_{i\in\cI_{02}\cup\cI_2}\vv_i^2},\rho\big\}$, we obtain
$\Phi_{\mu}(\wvx + \epsilon\vv) < \Phi_{\mu}(\wvx)$.

Therefore, either $\wvx + \epsilon\vv$ or $\wvx - \epsilon\vv$ decreases $\Phi_{\mu}(\wvx)$, 
which contradicts  the optimality of $\wvx$. Thus, $|\wvx_i| \neq \mu$ for all $i$.  
\end{proof}

Drawing inspiration from the work in \cite{nikolova2000thresholding}, we prove that, without any assumption on $A$, the objective function $\mathcal{G}_{\mu}(\vx)$ of model \eqref{obj_singleu} is differentiable in a neighborhood of its minimizer.

\begin{theorem}\label{Gfuncontinus}
Let $A \in \mathbb{R}^{m \times n}$ with $m < n$, and consider the function $\cG_{\mu}(\vx) = \tfrac{\alpha}{2} \norm{A\vx - \vb}_2^2 + \Phi_{\mu}(\vx)$ defined in \eqref{obj_singleu}. Then,
$\cG_{\mu}(\vx)$ is differentiable in a neighborhood of the optimal solution $\widehat{\vx} \in \mathbb{R}^n$ of \eqref{obj_singleu}. 
Furthermore,
\[
|\widehat{\vx}_i|  \neq \mu \quad \text{for all } i = 1, \ldots, n.
\]
\end{theorem}
\begin{proof}
Suppose, by contradiction, that $\widehat{\vx}_i = \mu$ (or $-\mu$) for some $i$.
Let $\ve_i\in\rr^n$ be the $i$-th unit vector. 
We show that moving $\wvx$ along $\pm\ve_i$, decreases $\cG_{\mu}(\vx)$, contradicting the optimality of $\wvx$. 

For $0 < \epsilon < 2\mu$, we have $-\mu < \widehat{\vx}_i - \epsilon < \widehat{\vx}_i = \mu < \widehat{\vx}_i + \epsilon$ . Thus, 
\[
\phi_{\mu}(\widehat{\vx}_i + \epsilon) = \phi_{\mu}(\widehat{\vx}_i) = 1, \quad \phi_{\mu}(\widehat{\vx}_i - \epsilon) = \tfrac{1}{\mu^2}(\widehat{\vx}_i - \epsilon)^2.
\]
Consequently, the $TH$ penalty satisfies
\[
\Phi_{\mu}(\widehat{\vx} + \epsilon\ve_i) = \Phi_{\mu}(\widehat{\vx}), \quad \Phi_{\mu}(\widehat{\vx} - \epsilon\ve_i) = \Phi_{\mu}(\widehat{\vx}) + \tfrac{1}{\mu^2}(\epsilon^2 - 2\widehat{\vx}_i\epsilon).
\]
Denoting $\va_i = A\ve_i$, the objective function $\mathcal{G}_{\mu}$ changes along the direction $\pm\ve_i$ is:
\[
\cG_{\mu}(\widehat{\vx} + \epsilon\ve_i) = \cG_{\mu}(\widehat{\vx}) + \epsilon\alpha\va_i^\top(A\widehat{\vx} - \vb) + \tfrac{1}{2}\alpha\norm{\va_i}_2^2\epsilon^2,
\]
and
\[
\cG_{\mu}(\widehat{\vx} - \epsilon\ve_i) = \cG_{\mu}(\widehat{\vx}) - \epsilon\left(\alpha\va_i^\top(A\widehat{\vx} - \vb) + \tfrac{2\widehat{\vx}_i}{\mu^2}\right) + \left(\tfrac{\alpha\norm{\va_i}_2^2}{2} + \tfrac{1}{\mu^2}\right)\epsilon^2.
\]

We analyze two cases based on the sign of $\va_i^\top(A\widehat{\vx} - \vb)$:
\begin{enumerate}[1)]
\item If $\va_i^\top(A\widehat{\vx} - \vb) \geq 0$, choosing $\epsilon < \min\left(\tfrac{2\left(\alpha\mu^2\va_i^\top(A\widehat{\vx} - \vb) + 2\widehat{\vx}_i\right)}{\alpha\mu^2\norm{\va_i}_2^2 + 2}, 2\mu\right)$ ensures that
\[
\cG_{\mu}(\widehat{\vx} - \epsilon\ve_i) < \cG_{\mu}(\widehat{\vx}).
\]

\item If $\va_i^\top(A\widehat{\vx} - \vb) < 0$, selecting $\epsilon < \min\left(-\tfrac{2}{\norm{\va_i}_2^2}\va_i^\top(A\widehat{\vx} - \vb), 2\mu\right)$ guarantees
\[
\cG_{\mu}(\widehat{\vx} + \epsilon\ve_i) < \cG_{\mu}(\widehat{\vx}).
\]
\end{enumerate}

These inequalities show that if $\wvx_i = \mu$ (similarly for $- \mu$), then $\wvx$ cannot be a minimizer of $\mathcal{G}_{\mu}(\cdot)$, as there exists a descent direction along $\pm \ve_i$. This contradiction implies that $|\widehat{\vx}_i| \neq \mu$ for all $i$ and $\cG_{\mu}(\vx)$ is differentiable around $\wvx$.
\end{proof}

\subsection{Optimality conditions}
Despite the non-differentiability of $\Phi_{\mu}(\vx)$ at the breakpoints $ \pm\mu$, Theorems \ref{THdifferentiable_refined} and \ref{Gfuncontinus} ensure that any minimizer satisfies $\abs{\wvx_i} \neq \mu$ for all $i$. 
Based on these results, we now derive the optimality conditions for both the noise-free and noisy models, formulated in Theorems \ref{linearoptcond} and \ref{noise_optcond}, respectively.


\begin{theorem}\label{linearoptcond}
Assume that any $(n-1)$-column submatrix of $A \in \mathbb{R}^{m \times n}$, $m < n$, has rank $m$. A vector $\widehat{\vx}$ is a local minimizer of \eqref{constrain_model},
if and only if $|\widehat{\vx}_i| \neq \mu$ for all $i$ and there exists $\widehat{\mathbf{y}} \in \mathbb{R}^m$ such that
\[
\tfrac{2}{\mu^2} (\mathbf{1} - \widehat{\mathbf{\vomega}}) \circ \widehat{\vx} + A^\top \widehat{\mathbf{y}} = \mathbf{0}, \quad \text{and} \quad A \widehat{\vx} = \vb,
\]
where $\widehat{\mathbf{\vomega}} = \mathcal{H}_{\mu}(\widehat{\vx})$ as defined in \eqref{porpoptw}. 
\end{theorem}

\begin{proof}
Let $\widehat{\vx}$ be a local minimizer of \eqref{constrain_model}. By Theorem \ref{THdifferentiable_refined}, $|\widehat{\vx}_i| \neq \mu$ for all $i$. 
From Lemma \ref{GradPhiTH}, $\Phi_{\mu}$ is differentiable in a neighborhood of the optimal solution $\widehat{\vx}$ with gradient:
\[
\nabla \Phi_{\mu}(\widehat{\vx}) = \tfrac{2}{\mu^2} (\mathbf{1} - \widehat{\mathbf{\vomega}}) \circ \widehat{\vx}.
\]

For any $\vv \in \ker(A)$ and scalar $t$, 
the path $\vx(t) = \widehat{\vx} + t \vv$ is feasible since $A \vx(t) = A \widehat{\vx} = \vb$. 
As $\widehat{\vx}$ is a local minimizer, there exists $\epsilon > 0$ such that $\Phi_{\mu}(\widehat{\vx} + t \vv) \geq \Phi_{\mu}(\widehat{\vx})$ for $|t| < \epsilon$. Thus, $t = 0$ minimizes $\phi(t) := \Phi_{\mu}(\widehat{\vx} + t \vv)$, and for all $\vv \in \ker(A)$
\[
\phi'(0) = \nabla \Phi_{\mu}(\widehat{\vx})^\top \vv = 0.
\]
Hence, $\nabla \Phi_{\mu}(\widehat{\vx}) \perp \ker(A)$. Since $\ker(A)^\perp = \text{range}(A^\top)$ (as $A$ has full row rank), there exists $\widehat{\mathbf{y}} \in \mathbb{R}^m$ such that:
\[
\nabla \Phi_{\mu}(\widehat{\vx}) = - A^\top \widehat{\mathbf{y}}.
\]
Substituting $\nabla \Phi_{\mu}(\widehat{\vx})= \tfrac{2}{\mu^2} (\mathbf{1} - \widehat{\mathbf{\vomega}}) \circ \widehat{\vx}$, we obtain:
\[
\tfrac{2}{\mu^2} (\mathbf{1} - \widehat{\mathbf{\vomega}}) \circ \widehat{\vx} + A^\top \widehat{\mathbf{y}} = \mathbf{0}.
\]
Since $A \widehat{\vx} = \vb$, the necessary conditions hold for local  minimizers.

For sufficiency, assume that $\wvx$ satisfies $\abs{\wvx_i}\neq\mu$ for all $i$,  $\tfrac{2}{\mu^2}(\mathbf1-\widehat{\vomega})\circ \wvx + A^\top \wvy = \mathbf0$, and $A\wvx=\vb$.
Define
\[
0<\rho=\tfrac{1}{2}\min_i\{\abs{\abs{\wvx_i}-\mu} \},\quad \cI = \{i\mid\abs{\wvx_i}<\mu\}.
\]
Let $\widetilde{\vx}$ be any feasible solution satisfying $A\widetilde{\vx}=\vb$ and $\norm{\widetilde{\vx}-\widehat{\vx}}_\infty\leq \rho$. Denote $\vv=\widetilde{\vx}-\widehat{\vx}$. Then we have $\vv\in\ker(A)$, and the following holds:
\[
\phi_{\mu}(\widetilde{\vx}_i)=\left\{
\begin{array}{ll}
1, &\mathrm{if}\;\abs{\wvx_i}>\mu, \\
\tfrac{1}{\mu^2}(\wvx_i+\vv_i)^2,&\mathrm{if}\;\abs{\wvx_i}<\mu.
\end{array}
\right.
\]
The difference $\Phi_{\mu}(\widetilde{\vx})-\Phi_{\mu}(\wvx)$ is given by
\[
\Phi_{\mu}(\widetilde{\vx})-\Phi_{\mu}(\wvx) = \tfrac{1}{\mu^2}\sum_{i\in\cI }\left((\wvx_i+\vv_i)^2-\wvx_i^2\right)=\tfrac{1}{\mu^2}\sum_{i\in\cI }\left(\vv_i^2+2\wvx_i\vv_i\right).
\]
Using $(A\vv)^\top \wvy=0$ and $\sum_{i\in\cI }\wvx_i\vv_i=\vv^\top ((\mathbf1-\widehat{\vomega})\circ \wvx )=\vv^\top ((\mathbf1-\widehat{\vomega})\circ \wvx +\tfrac{\mu^2}{2}A^\top \wvy)=0$,
we deduce:
\[
\Phi_{\mu}(\widetilde{\vx})-\Phi_{\mu}(\wvx) = \tfrac{1}{\mu^2}\sum_{i\in\cI}\vv_i^2\geq0.
\]
This inequality confirms that $\wvx$ is indeed a local minimizer.

\end{proof}

Theorem \ref{linearoptcond} requires the assumption that every $(n-1)$-column submatrix of $A$ has full row rank. We will assume that the matrix $A$ satisfies this property throughout the sequel.
Following \cite[Proposition 2]{nikolova2000thresholding}, we can easily {prove} the optimality condition for the noisy case.
\begin{theorem}\label{noise_optcond}
A vector $\widehat{\vx}$ is a local minimizer of \eqref{obj_singleu}  if and only if $\abs{\wvx_i}\neq \mu$ for all $i$, and there exists $\widehat{\vomega}\in\rr^n$ such that 
\[
\alpha A^\top(A\wvx-\vb) + \tfrac{2}{\mu^2}(\mathbf1-\widehat{\vomega})\circ \wvx = \mathbf0,
\]
where $\widehat\vomega=\cH_{\mu}(\widehat\vx)$, as defined in \eqref{porpoptw}. 
\end{theorem}

\section{Block coordinate descent method}\label{BCDmethod}

This section develops numerical methods to solve the proposed non-convex optimization problems \eqref{constrain_model} and \eqref{obj_singleu}. The numerical schemes rely on the minimization of a surrogate function, which will be introduced in section \ref{surgatefunc}.

\subsection{Surrogate function}\label{surgatefunc} 

The non-convexity of the $TH$ penalty $\Phi_{\mu}(\vx)$ complicates the direct application of optimality conditions. 
To address this issue, we introduce a surrogate function $\cQ_\mu(\vx, \vomega)$ that reformulates the original problem into a more tractable bi-variate optimization problem.
Structurally, $\cQ_\mu(\vx, \vomega)$ takes the form:
\begin{equation}\label{OptTransfun_rev}
\cQ_\mu(\vx, \vomega) = \tfrac{1}{\mu^2} \norm{(\mathbf{1} - \vomega) \circ \vx}_2^2 + \norm{\vomega}_0,
\end{equation}
where $\vomega \in \mathbb{R}^n$ is an auxiliary variable, and $\|\vomega\|_0$ counts the number of non-zero entries in $\vomega$. 
Then, the surrogate function satisfies $\Phi_{\mu}(\vx) \leq \cQ_\mu(\vx, \vomega)$  for any $\vx\in\rr^n.$
 
\begin{lemma}
\label{upperbound_lema2_rev}
For $\cQ_\mu(\vx, \vomega)$ defined in \eqref{OptTransfun_rev}, the properties  
\begin{equation}\label{surrogate_lemma}
    \cH_{\mu}(\vx) = \argmin_{\vomega\in\rr^n}\cQ_\mu(\vx, \vomega),\quad\text{and}\quad \Phi_{\mu}(\vx) = \min_{\vomega\in\rr^n} \cQ_\mu(\vx, \vomega),  \end{equation}
hold. Here, $\cH_{\mu}(\cdot)$ is defined in \eqref{porpoptw}.
\end{lemma}

\begin{proof}
The function $\cQ_\mu(\vx, \vomega)$ is separable in $\vomega$, i.e.,  
\[
\cQ_\mu(\vx, \vomega) = \sum_{i=1}^{n}\left(\tfrac{1}{\mu^2} (\vx_i (\vomega_i - 1))^2 + \abs{\vomega_i}^0\right),
\]
where $\abs{\vomega_i}^0$ is the indicator function that equals $1$ when $\vomega_i \neq 0$ and $0$ otherwise.
This separability allows us to solve $\min_{\vomega}\cQ_\mu(\vx, \vomega)$ by minimizing over each $\vomega_i$ independently. The component-wise minimization yields:   If $\vx_i = 0$, the minimum is achieved at $\widehat\vomega_i = 0$.
If $\vx_i \neq 0$, the subproblem reduces to  
\[
    \min_{\vomega_i} \left( (\vomega_i - 1)^2 + \tfrac{\mu^2}{\vx_i^2} |\vomega_i|^0 \right).  
\]  
The optimal solution follows a hard thresholding rule:  
\[
    \widehat{\vomega}_i = \begin{cases}  
    0, & |\vx_i| \leq \mu, \\  
    1, & |\vx_i| > \mu,  
    \end{cases}  
\]  
which corresponds to $\big(\cH_{\mu}(\vx)\big)_i$, as defined in \eqref{porpoptw}.

Combining both cases, we have $\widehat{\vomega} = \cH_{\mu}(\vx) = \arg\min_{\vomega} \cQ_\mu(\vx, \vomega)$. Substituting $\widehat{\vomega}$ into $\cQ_\mu(\vx, \vomega)$ leads to \eqref{surrogate_lemma}.
Thus the theorem holds.
\end{proof}

This result reformulates the problem $\min_{\vx \in \cS} \Phi_{\mu}(\vx)$  (similarly for $\min\cG_{\mu}(\vx)$) into the bi-variate optimization problem $\min_{\vx \in \cS, \vomega\in\rr^n} \cQ_\mu(\vx, \vomega)$, enabling more efficient optimization.

To solve the bivariate optimization problem, we adopt the block coordinate descent (BCD) method \cite{tseng2001convergence}. 
This iterative approach alternately optimizes $\vx$ and $\vomega$, with the other variable fixed, ensuring a monotonic decrease in the objective function. 
{
The subproblem for $\vomega$ is separable, allowing decomposition into independent univariate optimization problems, while the subproblem for $\vx$ admits a closed-form solution. 
}
This structure enables the BCD method to efficiently identify the minimizer. The following subsections develop tailored numerical algorithms within this framework for both noise-free and noisy scenarios.

\subsection{Numerical scheme for noise-free case}\label{noiselessalgsection}

For the noise-free case, we consider the constrain model \eqref{constrain_model}. 
According to Lemma \ref{upperbound_lema2_rev}, we transform this model into the following bi-variate optimization problem: 
\begin{equation}\label{BCDconstrain_prob} 
(\wvx,\widehat{\vomega})=\arg\min_{A\vx=\vb,\vomega}\cQ_\mu(\vx, \vomega).
\end{equation}
At the $j+1$-th iteration, the BCD scheme is: 
\begin{equation}\label{BCDschemefornoiseless}
\left\{
\begin{aligned}
&\wvx^{j+1} \in\arg\min _{A\vx=\vb }\cQ_\mu(\vx,\widehat{\vomega}^{j}),\\
&\widehat{\vomega}^{j+1}\in\arg\min_{\vomega}\cQ_\mu(\wvx^{j+1},\vomega). 
\end{aligned}
\right.
\end{equation}
Assume that $\mathrm{spark}(A)=m+1$ with $A\in\rr^{m\times n}$ ($m\ll n$), and select $\cQ_{\mu}(\wvx^{0},\widehat\vomega^{0})<\mathrm{spark}(A)$.
For the $\vx$-subproblem, the corresponding Lagrangian is:
\[
\cL_1(\vx,\vy;\mu)=\tfrac{1}{\mu^2}\LS{(\mathbf{1}-\widehat\vomega^{j})\circ\vx}+\norm{\widehat\vomega^{j}}_0+\vy^\top(A\vx-\vb),
\]
where $\vy$ is the Lagrangian multiplier.
By Kuhn-Tucker conditions (KKT) \cite{boyd2004convex}, there exists $\wvy^{j+1}$ such that
\begin{equation}\label{opticonbcd}
    \tfrac{2}{\mu^2}(1 - \widehat{\vomega}^{j}) \circ \wvx^{j+1} + A^\top \wvy^{j+1} = \mathbf0,~\text{and}~ A\wvx^{j+1}=\vb,
\end{equation}
where $\widehat{\vomega}^{j}\in\arg\min_{\vomega}\cQ_{\mu}(\wvx^{j},\vomega)$ is a binary vector, according to Lemma \ref{upperbound_lema2_rev}.
Let $\widehat\vomega^{j}_{i}$ be the $i$-th element of $\widehat{\vomega}^{j}$ and denote $A_\cI$ as the submatrix of $A$ containing the columns indexed by $\cI\subset\{1,\ldots,n\}$. 
Partition $A$ into $A_{\cI^{j}_{1}}$ and $A_{\cI^{j}_{2}}$ with 
\begin{equation}\label{indexfornoiseless}
   \cI^{j}_{1} = \{i \mid  \widehat\vomega^{j}_{i} = 1,~i=1,\ldots,n\},~\cI^{j}_{2} = \{i \mid  \widehat\vomega^{j}_{i} = 0,~i=1,\ldots,n\}.
\end{equation}
The condition \eqref{opticonbcd} leads to
\[
A_{\cI^{j}_1}^\top \wvy^{j+1} = \mathbf{0}, \quad A_{\cI^{j}_2}^\top \wvy^{j+1} + \tfrac{2}{\mu^2} \wvx_{\cI^{j}_2}^{j+1} = \mathbf{0}.  
\]  
From \eqref{BCDschemefornoiseless}, it follows that 
\[
\norm{\widehat\vomega^{j}}_0\leq\cQ_{\mu}(\wvx^{j},\widehat\vomega^{j})\leq\cQ_{\mu}(\wvx^{j},\widehat\vomega^{j-1})\leq\cQ_{\mu}(\wvx^{j-1},\widehat\vomega^{j-1})\leq\cdots\leq\cQ_{\mu}(\wvx^{0},\widehat\vomega^{0}).
\]
Thus, $\mathrm{card}(\cI_1^j)=\norm{\widehat\vomega^{j}}_0\leq m$ and $\mathrm{card}(\cI_2^j)\geq (n-m)>m$. 
Since $m=\mathrm{spark}(A)-1\leq\mathrm{rank}(A)$, the matrices $A_{\cI^{j}_{1}}^\top A_{\cI^{j}_{1}}$ and $A_{\cI^{j}_{2}}A_{\cI^{j}_{2}}^\top$ are invertible.
Using the constraint $A_{\cI^{j}_1} \wvx_{\cI^{j}_1}^{j+1} + A_{\cI^{j}_2} \wvx_{\cI^{j}_2}^{j+1} = \vb$, the update of $\wvx^{j+1}$ is: 
\begin{align}
\label{updateforx_I1}
\wvx_{\cI^{j}_{1}}^{j+1} &= \left(A_{\cI^{j}_{1}}^\top\left(A_{\cI^{j}_{2}}A_{\cI^{j}_{2}}^\top\right)^{-1}A_{\cI^{j}_{1}}\right)^{-1}A_{\cI^{j}_{1}}^\top\left(A_{\cI^{j}_{2}}A_{\cI^{j}_{2}}^\top\right)^{-1}\vb, \\
\label{updateforx_I2}
\wvy^{j+1} &= \tfrac{2}{\mu^2}\left(A_{\cI^{j}_{2}}A_{\cI^{j}_{2}}^\top\right)^{-1}(A_{\cI^{j}_{1}}\wvx_{\cI^{j}_{1}}^{j+1} - \vb),\\
\label{formulation_y}
\wvx_{\cI^{j}_{2}}^{j+1} &= -\tfrac{\mu^2}{2}A_{\cI^{j}_{2}}^\top \wvy^{j+1}.
\end{align}
For the $\vomega$-subproblem, the update at the $j+1$-th iteration is:
\begin{equation}\label{updateforw}
    \widehat{\vomega}^{j+1} = \cH_{\mu}(\wvx^{j+1}),
\end{equation}
where $\cH_\mu$ is defined in \eqref{porpoptw}.

The complete algorithm for the noise-free case is summarized in Algorithm \ref{BCD}.
\begin{algorithm}
\caption{BCD-TH sparse signal recovery}
\label{BCD}
\renewcommand{\algorithmicrequire}{\textbf{Input:}}
\renewcommand{\algorithmicensure}{\textbf{Output:}}
\begin{algorithmic}[1]
\REQUIRE $A$, $\vb$, threshold $\mu$, $\alpha$
\ENSURE $(\wvx,\widehat{\vomega})$   
\STATE  Initialize $(\wvx ^0, \widehat\vomega^0)$ and set  $j=0$
\WHILE{stopping criterion \eqref{stopcondition} is not satisfied}
\STATE Compute $\cI_{1}^j$ and $\cI_{2}^j$ from \eqref{indexfornoiseless}
\STATE Update $\wvx^{j+1}$ by \eqref{updateforx_I1}-\eqref{updateforx_I2} [noise-free \eqref{constrain_model}], or by \eqref{noiseupdateforx_I1}-\eqref{noiseupdateforx_I2} [noisy \eqref{obj_singleu}]
\STATE Update $\widehat{\vomega}^{j+1}$ using \eqref{updateforw}
\STATE $j=j+1$
\ENDWHILE
\RETURN $(\wvx\leftarrow\wvx^{j+1},\widehat{\vomega}\leftarrow\widehat{\vomega}^{j+1})$
\end{algorithmic}
\end{algorithm}

\subsection{Numerical scheme for the noisy case}
In the presence of noise, we consider the minimization problem \eqref{obj_singleu}. 
By Lemma \ref{upperbound_lema2_rev}, this problem can be reformulated as:  
\begin{equation}\label{thb_regularization}
(\wvx,\widehat{\vomega})=\arg\min_{\vx,\vomega}\cJ_\mu(\vx,\vomega) := \frac{\alpha}{2}\norm{A\vx - \vb}_2^2+\cQ_\mu(\vx, \vomega).
\end{equation}

The BCD scheme at the $j+1$-th iteration follows 
\begin{equation}\label{phi0bcmobj}
\left\{
\begin{aligned}
&\wvx^{j+1}\in\arg\min_{\vx }\cJ_\mu(\vx,\widehat{\vomega}^{j}),\\
&\widehat\vomega^{j+1}\in\arg\min_{\vomega}\cQ_\mu(\wvx^{j+1}, \vomega)
\end{aligned}
\right.
\end{equation}
The update for $\vomega$ remains the same as in the noise-free case, given by \eqref{updateforw}. 
For the $\vx$-subproblem, {
the solution satisfies:
\[
\left( \alpha A^\top A + \frac{2}{\mu^2} \operatorname{diag}(\mathbf{1} - \widehat{\vomega}^j) \right) \vx^{j+1} = \alpha A^\top \vb. 
\]
To reduce the computational complexity of solving this $n \times n$ linear system, particularly when $n$ is large, we introduce an auxiliary variable $\vz = A\vx$ to reformulate the problem as a smaller system. This yields the equivalent optimization problem:
}
\[
\min_{\vz=A\vx} \tfrac{\alpha}{2}\LS{\vz-\vb}+\tfrac{1}{\mu^2}\LS{(\mathbf1-\widehat\vomega^j)\circ\vx}+\norm{\widehat\vomega^j}_0.
\]
The corresponding Lagrangian function is
\[
    \cL_2(\vx,\vz,\vy;\mu)=\tfrac{\alpha}{2}\LS{\vz-\vb}+\tfrac{1}{\mu^2}\LS{(\mathbf{1}-\widehat\vomega^j)\circ\vx}+\vy^\top(A\vx-\vz),
\]
where $\vy\in\rr^m$ is the Lagrange multiplier.
By the optimality conditions, there exist $\wvy^{j+1}$ such that
\begin{equation}\label{opticonoisebcd}
    \tfrac{2}{\mu^2}(1 - \widehat{\vomega}^{j}) \circ \wvx^{j+1} + A^\top \wvy^{j+1} = \mathbf0,~\text{and}~ A\wvx^{j+1}=\vb+\tfrac{1}{\alpha}\widehat\vy^{j+1}.
\end{equation}
Following the same partitioning \eqref{indexfornoiseless} as in the noise-free case, 
we obtain the index sets $\cI_1^j$ and $\cI_2^j$.
This leads to the updated rules
\begin{align}
\label{noiseupdateforx_I1}
\wvx_{\cI^{j}_{1}}^{j+1} &= \left(A_{\cI^{j}_{1}}^\top\left(\tfrac{1}{\alpha}\mathrm{I}+\tfrac{\mu^2}{2}A_{\cI^{j}_{2}}A_{\cI^{j}_{2}}^\top\right)^{-1}A_{\cI^{j}_{1}}\right)^{-1}A_{\cI^{j}_{1}}^\top\left(\tfrac{1}{\alpha}\mathrm{I}+\tfrac{\mu^2}{2}A_{\cI^{j}_{2}}A_{\cI^{j}_{2}}^\top\right)^{-1}\vb, \\
\label{noiseupdateforx_I2}
\wvx_{\cI^{j}_{2}}^{j+1} &= -\tfrac{\mu^2}{2}A_{\cI^{j}_{2}}^\top \wvy^{j+1},
\end{align}
where $\wvy^{j+1} = \left(\tfrac{1}{\alpha}\mathrm{I}+\tfrac{\mu^2}{2}A_{\cI^{j}_{2}}A_{\cI^{j}_{2}}^\top\right)^{-1}(A_{\cI^{j}_{1}}\wvx_{\cI^{j}_{1}}^{j+1} - \vb).$
The detailed procedure is also summarized in Algorithm \ref{BCD}.

\subsection{Convergence analysis}\label{convergenceanly}
This subsection establishes the convergence properties of the proposed Algorithm \ref{BCD} in both noise-free and noisy settings. 
The analysis relies on the uniqueness and boundedness of the solution of a linear system, as well as the monotonicity of the objective function.
\begin{theorem}
\label{thm:uniqueness_boundedness}
Let $\lambda\in\rr_{++}$, $\vc \in \mathbb{R}^n$, $A \in \mathbb{R}^{m \times n}$ ($m\ll n$), and $D\in\rr^{n\times n}$ be a diagonal matrix with diagonal entries $D_{ii} \in \{0, 1\}$. 
{
Define $\mathcal{I} = \{i \mid D_{ii} = 0, i=1,\dots,n\}$. If $\mathrm{card}(\mathcal{I}) < \mathrm{spark}(A)$, then the solution $\vx$ to the linear system
}
\[
(A^\top A + \lambda D)\vx = \vc,
\]
is unique and bounded.
\end{theorem}
\begin{proof}
{
Notice that the spark of a matrix $ A $ is the smallest number of columns that are linearly dependent. For $ A \in \mathbb{R}^{m \times n} $ with $ m \ll n $, $\mathrm{spark}(A) \leq m + 1$, and if $ A $ has full row rank, $\mathrm{spark}(A) = m + 1$.
}
The assumption $\mathrm{card}(\cI)<\mathrm{spark}(A)$ ensures that $A_{\cI}$ is full column rank, implying that $A_{\cI}^\top A_{\cI}$ is invertible.

Reordering indices, the linear system can be rewritten as
\[
\begin{pmatrix}
A_{\cI}^\top A_{\cI} & A_{\cI}^\top A_{\cI^c} \\
A_{\cI^c}^\top A_{\cI} & A_{\cI^c}^\top A_{\cI^c} + \lambda  I 
\end{pmatrix}
\begin{pmatrix}
\vx_{\cI} \\ \vx_{\cI^c}
\end{pmatrix}
=
\begin{pmatrix}
\vc_{\cI} \\ \vc_{\cI^c}
\end{pmatrix}
,
\]
where $\cI^c$ is the complement of $\cI$.
To analyze the invertibility of the coefficient matrix, we consider the Schur complement:
\[
\begin{pmatrix}
  A_{\cI}^\top A_{\cI} &  A_{\cI}^\top A_{\cI^c} \\
  A_{\cI^c}^\top A_{\cI} &   A_{\cI^c}^\top A_{\cI^c} + \lambda  I 
\end{pmatrix}=\begin{pmatrix}
 I & 0 \\
  A_{\cI^c}^\top A_{\cI}\big(A_{\cI}^\top A_{\cI} \big)^{-1} &  I
\end{pmatrix}\begin{pmatrix}
  A_{\cI}^\top A_{\cI} &  A_{\cI}^\top A_{\cI^c} \\
 0  &   B
\end{pmatrix},
\]
where 
$
B =  \lambda  I + A_{\cI^c}^\top ( I  - A_{\cI}(A_{\cI}^\top A_{\cI})^{-1}A_{\cI}^\top) A_{\cI^c} .
$
Since $A_{\cI^c}^\top ( I  - A_{\cI}(A_{\cI}^\top A_{\cI})^{-1}A_{\cI}^\top) A_{\cI^c}$ is  symmetric and positive semi-definite,
$B$ is positive definite and invertible.
Consequently, $(A^\top A + \lambda D)$ is invertible, ensuring a unique solution.
Moreover, the solution satisfies
\[
\vx = (A^\top A + \lambda D)^{-1} \vc .
\]
The norm of $\vx$ is bounded by
$\|\vx\|_2 \leq \|(A^\top A + \lambda D)^{-1}\|_2 \|\vc\|_2$.
Since $(A^\top A + \lambda D)$ is invertible, $\|(A^\top A + \lambda D)^{-1}\|_2$ is finite, implying that $\vx$ is bounded.  This completes the proof.
\end{proof}

We now establish the convergence of the Algorithm \ref{BCD} of the noise-free model \eqref{constrain_model}.

\begin{theorem}\label{covgergalg1}
Assume that $\mathrm{spark}(A)=m+1$ and the initial value $\cQ_\mu(\wvx^0,\widehat\vomega^0)< \mathrm{spark}(A)$. Let $\{(\wvx^{k},\widehat{\vomega}^{k})\}_k$ denote the sequence generated by Algorithm \ref{BCD}. Then, the sequence $\{(\wvx^{k},\widehat{\vomega}^{k})\}_k$ is bounded and converges to the minimizer of model \eqref{constrain_model} in a finite number of steps. 
\end{theorem}
\begin{proof}
For any $k\geq1$, the sequence of objective values satisfies 
\[\|\widehat\vomega^{k-1}\|_0\leq\cQ_\mu(\wvx^{k-1},\widehat\vomega^{k-1})\leq\cdots\leq\cQ_\mu(\wvx^0,\widehat\vomega^0)<\mathrm{spark}(A).
\]
This implies that $\mathbf{1}-\widehat\vomega^{k-1}$ has fewer than $\mathrm{spark}(A)$ zero entries.
    From the optimality condition \eqref{opticonbcd}, we have
    \[
    \big(A^\top A + \tfrac{2}{\mu^2}\mathrm{diag}(\mathbf{1}-\widehat\vomega^{k-1})\big)\wvx^k=A^\top\vb-A^\top\widehat\vy^{k}.
    \] 
Since $\mathrm{spark}(A)=m+1$,  the existence of $\widehat{\vy}^k$ is guaranteed (see \eqref{formulation_y}).
According to Theorem \ref{thm:uniqueness_boundedness} and the update rule in \eqref{updateforw}, the sequence $\{(\wvx^k,\widehat{\vomega}^k)\}_k$ is bounded. 
There exists a convergent subsequence $\{(\wvx^{k_j}, \widehat{\vomega}^{k_j})\}_j$ such that $\lim_{j \to \infty}(\wvx^{k_j}, \widehat{\vomega}^{k_j}) = (\wvx, \widehat{\vomega})$.
At the limit point, the following hold:
\begin{align*}
\big(A^\top A+\tfrac{2}{\mu^2}\mathrm{diag}(\mathbf{1}-\widehat\vomega)\big)\wvx&=A^\top \vb-A^\top\widehat\vy, \\
\widehat\vomega &= \cH_{\mu}(\wvx).
\end{align*}
Theorem  \ref{linearoptcond} establishes that $(\wvx, \widehat{\vomega})$ is a minimizer of the model \eqref{constrain_model}.
Due to the binary property of $\{\widehat{\vomega}^k\}_k$, there exists a sufficiently large $K$ such that for all $k_j > K$:  
\[
\widehat{\vomega}^{k_j} = \widehat{\vomega}\quad \mathrm{and}\quad \|\widehat{\vomega}\|_0 < \mathrm{spark}(A).
\]
Combined with \eqref{indexfornoiseless}, the index sets based on  
$\widehat{\vomega}^{k_j}$ matches those of the limit point $\widehat{\vomega}$. 
Using the update rules \eqref{updateforx_I1} and \eqref{updateforx_I2}, it follows that 
$$ \wvx^{k_j+1}=\wvx.$$
Consequently, we obtain
$$\widehat{\vomega}^{k_j+1} = \cH_{\mu}(\wvx^{k_j+1}) = \cH_{\mu}(\wvx) = \widehat{\vomega}.$$
Thus, $(\wvx^{k_j+1}, \widehat{\vomega}^{k_j+1}) = (\wvx, \widehat{\vomega})$. 
By induction,  for all $k>K$, we have $(\wvx^{k}, \widehat{\vomega}^{k}) = (\wvx, \widehat{\vomega})$.
This demonstrates that $\{(\wvx^k,\widehat{\vomega}^k)\}_k$ converges to the minimizer of model \eqref{constrain_model} in a finite number of steps.
\end{proof}

For the noisy case, we analyze the convergence of Algorithm \ref{BCD} of model \eqref{obj_singleu}.
Unlike Theorem \ref{covgergalg1}, this analysis does not require any assumption on $A$.

\begin{theorem}\label{convergalg2} 
Let the initial point $(\wvx^{0},\widehat{\vomega}^{0})$ satisfy $\cJ_\mu(\wvx^{0},\widehat{\vomega}^{0})< \operatorname{spark}(A)$. 
Let $\{(\wvx^{k},\widehat{\vomega}^{k})\}_k$ denote the sequence generated by  Algorithm \ref{BCD}. 
Then, the sequence $\{(\wvx^k,\widehat\vomega^k)\}_k$ is bounded and converges to the minimizer of model \eqref{obj_singleu} in a finite number of steps.
\end{theorem}

\begin{proof}
From \eqref{phi0bcmobj}, we obtain the following descent property:
\[
\cJ_\mu(\wvx^{k},\widehat{\vomega}^{k}) \leq \cJ_\mu(\wvx^{k},\widehat{\vomega}^{k-1}) \leq \cJ_\mu(\vx^{k-1},\widehat{\vomega}^{k-1})\leq\cdots\leq\cJ_\mu(\wvx^0,\widehat{\vomega}^0).
\]
At the $k$-th iteration, the optimality condition \eqref{opticonoisebcd} for $\wvx^k$ implies: 
\begin{equation}\label{conviterk}
    \alpha A^\top(A\wvx^k-\vb)+\tfrac{2}{\mu^2}(\mathbf{1}-\widehat{\vomega}^{k-1})\circ\wvx^k=\mathbf{0}.
\end{equation}
Since $\|\widehat{\vomega}^{k-1}\|_0\leq\cJ_\mu(\wvx^{k-1},\widehat{\vomega}^{k-1}) < \operatorname{spark}(A)$, 
the diagonal entries of $D:=\mathrm{diag}(\mathbf1-\widehat\vomega^{k-1})$ contain fewer than $\operatorname{spark}(A)$ zeros.
Setting $\lambda=\tfrac{2}{\alpha\mu^2}$ and  $c=A^\top \vb$,
Theorem \ref{thm:uniqueness_boundedness} guarantees that $\wvx^{k}$ is uniquely determined and bounded for all $k$.  
Therefore, there exists a convergent subsequence such that $\lim_{j\to\infty}(\wvx^{k_j},\widehat{\vomega}^{k_j})=(\wvx,\widehat{\vomega}).$ 
Passing to the limit in \eqref{conviterk}, we obtain
\[
\alpha A^\top(A\wvx-\vb)+\tfrac{2}{\mu^2}(\mathbf{1}-\widehat{\vomega})\circ\wvx=\mathbf{0}.
\]
According to Theorem \ref{noise_optcond}, the limit point $(\wvx,\widehat{\vomega})$ is a minimizer of model \eqref{obj_singleu}.
Similar to Theorem \ref{covgergalg1}, there exists a constant $K>0$ such that $(\wvx^k,\widehat{\vomega}^k)=(\wvx,\widehat{\vomega})$ for all $k>K$.
Thus, the entire sequence $\{(\wvx^k,\widehat{\vomega}^k)\}_k$ converges to $(\wvx,\widehat{\vomega})$ in a finite number of steps.
\end{proof}

\subsection{Continuation strategy  for $\mu$}\label{gradualmu}

In numerically solving the noise-free model \eqref{BCDconstrain_prob}, it is essential to ensure the invertibility of the matrices $A_{\cI^j_1}^\top A_{\cI^j_1}$ and $A_{\cI^j_2}A_{\cI^j_2}^\top$ (see \eqref{updateforx_I1}), which requires $\mathrm{card}(\cI^j_1) \leq m$. 
Similarly, for the noisy model \eqref{thb_regularization}, the matrix $A^\top A + \tfrac{2}{\alpha\mu^2} \text{diag}(\mathbf{1} - \widehat{\vomega}^{j})$ must remain invertible during iterations. 


To achieve this, we employ a continuation strategy starting with a large initial parameter $\mu_0$, gradually reducing $\mu_\ell$ toward the target value $\mu$ specified in \eqref{BCDconstrain_prob} or \eqref{thb_regularization}. 

The initial $\mu_0$ is selected based on the initial guess $\wvx^{0}$ to ensure 
that the submatrix $A_{\mathcal{I}_1^{0}}$ (where $\mathcal{I}_1^{0} = \{i : \widehat{\vomega}_i^{0} = 1\}$) has at most $m$ columns. Under the assumption in Theorem~2.7 that any $(n-1)$-column submatrix of $A$ is full rank, any submatrix with at most $m$ columns is full column rank. This guarantees the required invertibility and ensures numerical stability throughout the iterative process.

Next, we derive the update scheme for $\mu_{\ell+1}$. Specifically, given the initial value $(\wvx_{\ell}, \widehat{\vomega}_{\ell}, \mu_{\ell+1})$ in the $(\ell+1)$-th iteration, 
we aim to ensure that
$\cQ_{\mu_{\ell+1}}(\wvx_{\ell}, \widehat{\vomega}_{\ell}) < m+1$ (or $\cQ_{\mu_{\ell+1}}(\wvx_{\ell}, \widehat{\vomega}_{\ell}) \leq\cJ_{\mu_{\ell+1}}(\wvx_{\ell}, \widehat{\vomega}_{\ell}) < m+1$). 
According to Lemma \ref{mudecre_objincre} and Lemma \ref{upperbound_lema2_rev}, we have
\[
    \cQ_{\mu_{\ell+1}}(\wvx_{\ell}, \widehat{\vomega}_{\ell})\geq\Phi_{\mu_{\ell+1}}(\wvx_{\ell})\geq\Phi_{c_{\ell+1}}(\wvx_{\ell})=\cQ_{c_{\ell+1}}(\wvx_{\ell},\mathcal{H}_{c_{\ell+1}}(\wvx_{\ell})),
\]
where $\mu_{\ell+1}\leq c_{\ell+1} = \tfrac{\mu_{\ell}}{\rho}$ for some constant $\rho > 1$. 
The right-hand side can be expanded as
\[
\cQ_{c_{\ell+1}}(\wvx_{\ell}, \mathcal{H}_{c_{\ell+1}}(\wvx_{\ell})) =\frac{1}{c_{\ell+1}^2} \sum_{i \in \mathcal{K}} (\wvx_{\ell})_i^2 + n - \mathrm{card}(\mathcal{K}),
\]
with $\mathcal{K} = \{i \mid \abs{(\wvx_{\ell})_i}\leq c_{\ell+1} \}$.
To ensure $\cQ_{c_{\ell+1}}(\cdot) < m+1$, we require
\[
c_{\ell+1} > \sqrt{ \frac{ \sum_{i \in \mathcal{K}} (\wvx_{\ell})_i^2 }{ m - n + \mathrm{card}(\mathcal{K}) + 1 } }.
\]
To guarantee numerical stability and avoid over-updating, we clip this update between a lower bound $c_{\ell+1}$ and an upper bound $\mu_{\ell}$. The final update scheme becomes:
\begin{equation}\label{updatemu}
    \mu_{\ell+1} = \min\left\{ \max\left\{ \tfrac{\mu_{\ell}}{\rho}, \sqrt{ \frac{ \sum_{i \in \mathcal{K}} (\wvx_{\ell})_i^2 }{ m - n + \mathrm{card}(\mathcal{K}) + 1 } } \right\}, \mu_{\ell} \right\}.
\end{equation}

In Algorithm \ref{continuescheme}, we summarize the continuation procedure as an outer loop of the proposed Algorithm \ref{BCD}. 

\begin{algorithm}
\caption{Continuation strategy for BCD-TH penalized sparse signal recovery}
\label{continuescheme}
\renewcommand{\algorithmicrequire}{\textbf{Input:}}
\renewcommand{\algorithmicensure}{\textbf{Output:}}
\begin{algorithmic}[1]
\REQUIRE $A$, $\vb$, $\mu_0$, $\alpha$, $Max\mbox{-}epochs$
\ENSURE $(\wvx,\widehat{\vomega})$   
\STATE  Initialize $(\wvx ^0, \widehat\vomega^0)$ such that $\cQ_{\mu_0}(\wvx^0,\widehat\vomega^0)$ (or $\cJ_{\mu_0}(\wvx^{0},\widehat{\vomega}^{0})$) satisfies $\cQ_{\mu_0}(\wvx^0,\widehat\vomega^0)<\operatorname{spark}(A)$, and set  $\ell=0$
\WHILE{$\ell\leq Max\mbox{-}epochs$}
\STATE Compute $(\wvx_{\ell}, \widehat{\vomega}_{\ell})$ by applying steps 2–7 of Algorithm \ref{BCD}, using $(\wvx_{\ell-1}, \widehat{\vomega}_{\ell-1}, \mu_{\ell})$ as the initialization
\STATE Update $\mu_{\ell+1}$ by \eqref{updatemu}
\STATE $\ell=\ell+1$
\ENDWHILE
\RETURN $(\wvx\leftarrow\wvx_{\ell},\widehat{\vomega}\leftarrow\widehat{\vomega}_{\ell})$
\end{algorithmic}
\end{algorithm}

\section{Numerical experiments}\label{sec_experiments}
We now illustrate the performance of our proposed method ($TH$) for signal recovery problem. 
We compare the $TH$ method with the following  state-of-the-art sparse recovery approaches: 
$L_1$-minimization \cite{glowinski1975approximation}, $TL_1$ regularization \cite{zhang2018minimization}, $L_1\mbox{-}L_2$ composite regularization \cite{yin2015minimization},  IHT for $L_0$-constraints \cite{blumensath2009iterative}, $MCP$ optimization \cite{sun2018sparse}, IRLS implementation of $L_p$-minimization \cite{lai2013improved}, {ADMM for $L_1/L_2$ minimization \cite{tao2022minimization}, and DCA-based sorted $L_1/L_2$ \cite{wang2024sorted}.}
All simulations were implemented in MATLAB R2019b and executed on a MacBook Pro (1.4 GHz Intel Core i5 processor, 16 GB RAM) under macOS Big Sur (version 11.7).

The initialization process for all algorithms employs the $L_1$-norm minimization formulation:
$\vx^0 \in \arg\min_{A\vx=\vb} \|\vx\|_1$,
which is solved numerically using the Gurobi optimization package \cite{gurobi}. 
Parameter configurations for the benchmark methods are set as follows: the $TL_1$ method employed a parameter {$\beta=1$}, while the $L_p$-norm used   $p=0.5$. 
For the $TL_1$, $L_1\mbox{-}L_2$, and $MCP$ regularizers, we fix the regularization parameter at $\lambda=10^{-6}$ and implement the DCA for numerical optimization. 
The $L_0$-based approach adopts parameter settings from \cite{blumensath2009iterative} to maintain methodological consistency. 
For the {$L_p$, $L_1/L_2$, sorted $L_1/L_2$, and} $TH$ methods, the model parameter was manually selected to achieve the best numerical performance across experiments.
The iterative procedures for all methods terminate when either of the following criteria is satisfied:
\begin{equation}\label{stopcondition}
    \frac{\|\wvx^{j+1}-\wvx^{j}\|}{\|\wvx^{j}\|} < 10^{-8} \quad \text{or} \quad j > 5n,
\end{equation}
where $n$ denotes the signal dimension. Reconstruction performance was quantified through the relative recovery error (RRE):
\begin{equation}
    \text{RRE}(\wvx,\bar\vx)=\frac{\|\wvx - \overline{\vx}\|_2}{\|\overline{\vx}\|_2},
\end{equation}
with $\bar\vx$ representing the ground truth signal. 
The true signal $\overline{\vx}$ is an s-sparse vector supported on a random index set, with non-zero entries drawn from a Gaussian distribution, {except in subsection~\ref{exp_decay}, where they follow an exponential decay.} The measurement vector is obtained as $\vb=A\overline{\vx}$ for noise-free case and $\vb = A\overline{\vx} + \vn$ with the white Gaussian noise $\vn \sim \mathcal{N}(\mathbf{0}, \sigma^2 \mathbf{I}_m)$ for noise case.

We consider two types of sensing matrices:
\begin{enumerate}[1). ,leftmargin=*]
\item Gaussian matrix $A_1 \in \mathbb{R}^{m \times n}$, generated from normal distribution $\mathcal{N}(\mathbf{0}, \Sigma)$, where the covariance matrix is given by
    $\Sigma = \{(1 - r)\operatorname{I}_n(i=j) + r\}_{i,j}$
    with a positive parameter $r$. A larger $r$ indicates a more challenging recovery problem \cite{zhang2018minimization}.

    \item Oversampled partial discrete cosine (DCT) matrix $A_2 \in \mathbb{R}^{m \times n}$, formed as $A_2=[\va_1,\ldots,\va_n]$, with columns $\va_i$ constructed as follows
    \[
    \va_i = \frac{1}{\sqrt{m}} \cos\Bigl(\frac{2i\pi \gamma}{F}\Bigr),\quad i=1,\ldots,n,
    \]
    where $\gamma\in[0,1]^m$ is a random vector and $F>0$ is the refinement factor. A larger $F$ results in higher coherence and a more ill-conditioned matrix.
\end{enumerate}

\subsection{Convergence behavior under continuation strategy}

{In this subsection, we examine the convergence properties of the continuation method proposed in Section~\ref{gradualmu}. 
This approach employs an iterative strategy where the parameter $\mu$ is sequentially reduced through successive epochs. 
At each epoch, Algorithm~\ref{BCD} solves the optimization subproblem for the current $\mu_\ell$ before advancing to the next epoch. 
}

{Experiments are evaluated on both the Gaussian matrix $A_1 \in \mathbb{R}^{64 \times 512}$ ($r = 0.8, s = 12$) and the oversampled DCT matrix $A_2 \in \mathbb{R}^{64 \times 512}$ ($F = 5, s = 12$). The non-zero entries of ground truth $\bar{\vx}\in\rr^{512}$ are sampled from a standard Gaussian distribution.  Figure~\ref{muvsobjconstrained} shows the convergence behavior for both matrix types $A_1$ and $A_2$.  The left plots show the evolution of the regularization parameter $\mu$ and the objective function value versus the number of iterations, while the right plots show the evolution of RRE. 
}

{
We observe a consistent behavior in both matrix types:
\begin{enumerate}[1). ,leftmargin=*]
    \item For each fixed $\mu_{\ell}$ (corresponding to iterations 1-4, 5-7, 8-10, 11-13, and 14-15 for $A_1$; iterations 1-3, 4-6, 7-9, 10-14, 15-17, 18-19, 20-21, and 22-23 for $A_2$), the objective function value decreases monotonically or stabilizes, confirming finite step convergence.
    \item When $\mu$ is decreased (corresponding to iterations 4–5, 7–8, 10–11, and 13–14 for $A_1$; iterations 3-4, 6-7, 9-10, 14-15, 17-18, 19-20, and 21-22 for $A_2$), the objective function either increases (as observed in iterations 4–5, 7–8, and 10–11 for $A_1$; iterations 3-4, 6-7, 9-10, and  14-15 for $A_2$ ) or remains unchanged (as seen in iterations 13–14 for $A_1$, iterations 17-18, 19-20, and 21-22 for $A_2$).
    \item The RRE for both matrices drops sharply to machine precision ($\mathcal{O}(10^{-15})$) and then stabilizes. This occurs after a few epochs of reducing $\mu$, demonstrating that the continuation strategy effectively guides the iteration toward an accurate solution.
\end{enumerate}
}
{
These empirical findings are in agreement with Lemma~\ref{mudecre_objincre}, which states that the objective function value does not decrease as $\mu$ is reduced, and with Theorem~\ref{covgergalg1}, which guarantees finite-step convergence for a fixed $\mu$.
The consistent behavior across different types of sensing matrices ($A_1$ and $A_2$) validates the practical efficiency and robustness of the proposed continuation strategy.}


\begin{figure}
    \centering
    \setlength{\tabcolsep}{0.75em}
\begin{tabular}{cc}
\includegraphics[width=0.45\linewidth]{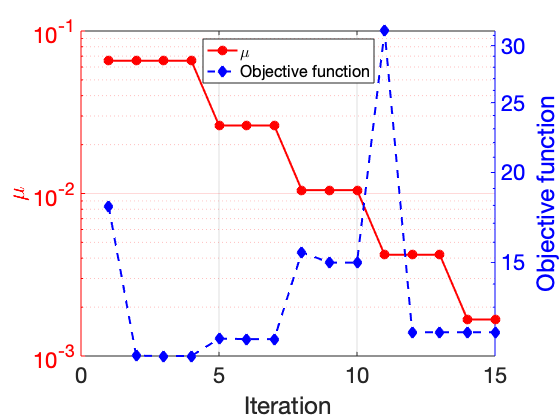} &
\includegraphics[width=0.45\linewidth]{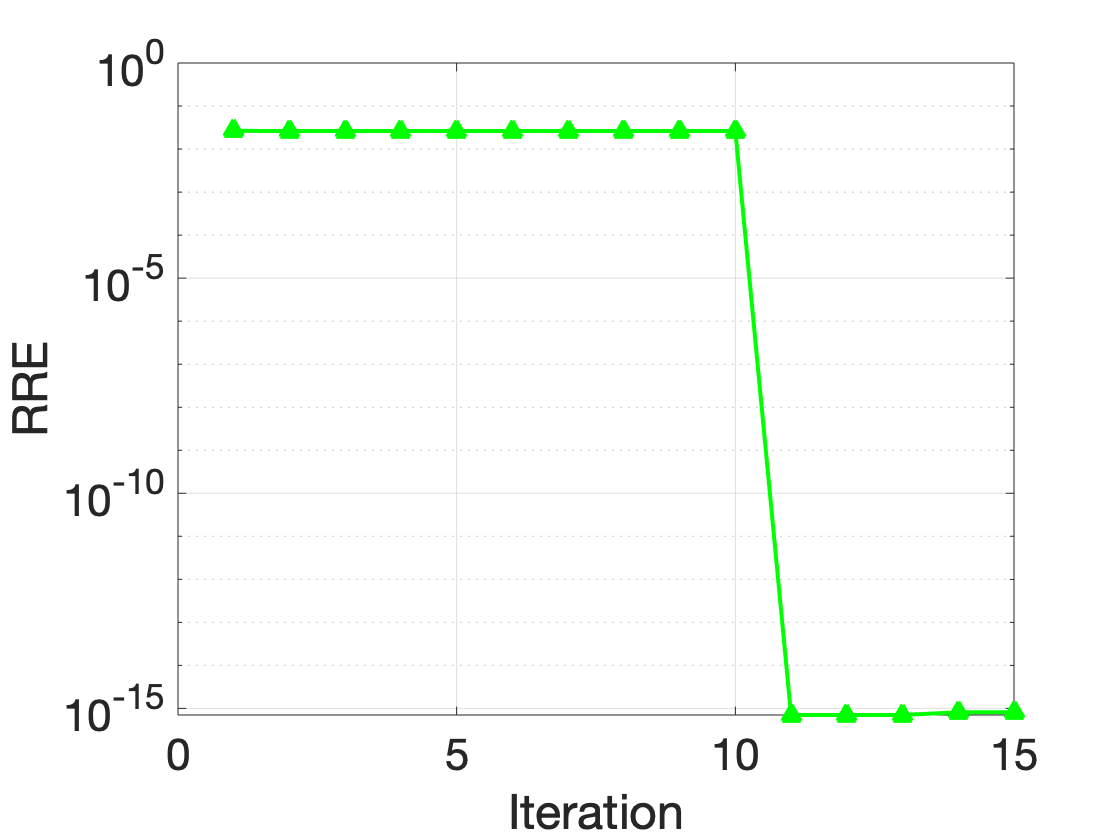} \\
\includegraphics[width=0.45\linewidth]{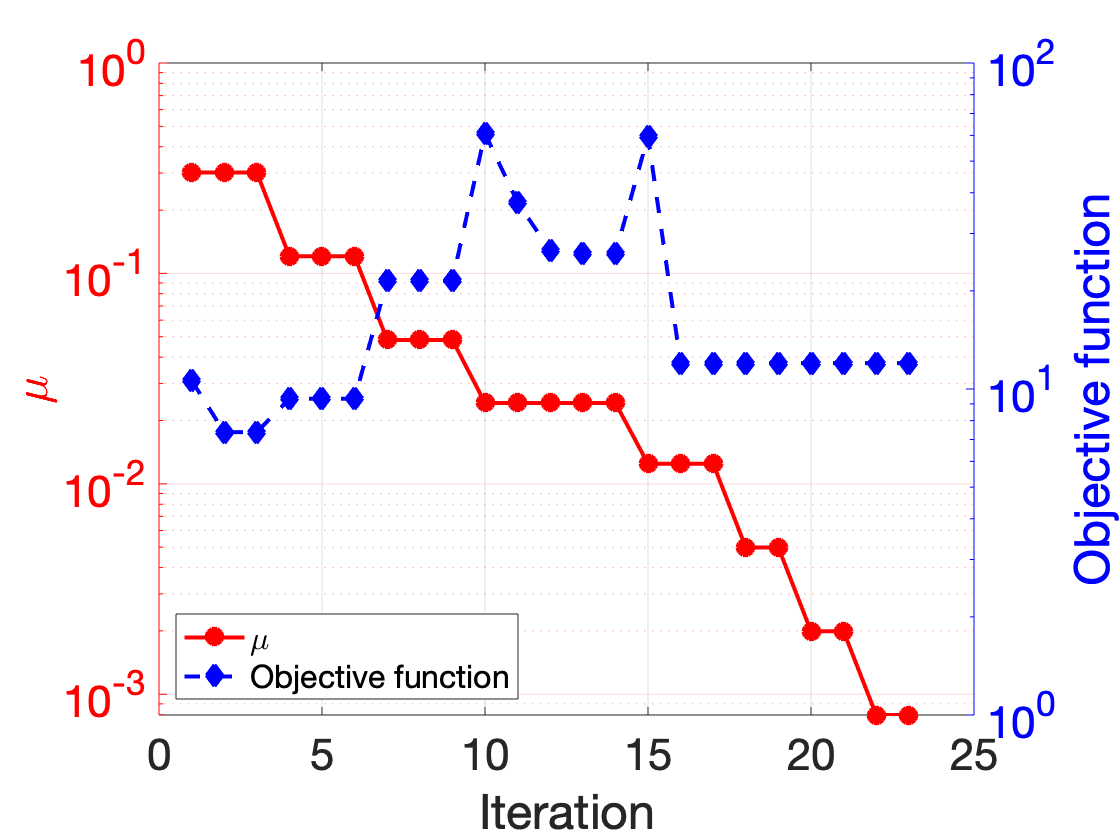} &
\includegraphics[width=0.45\linewidth]{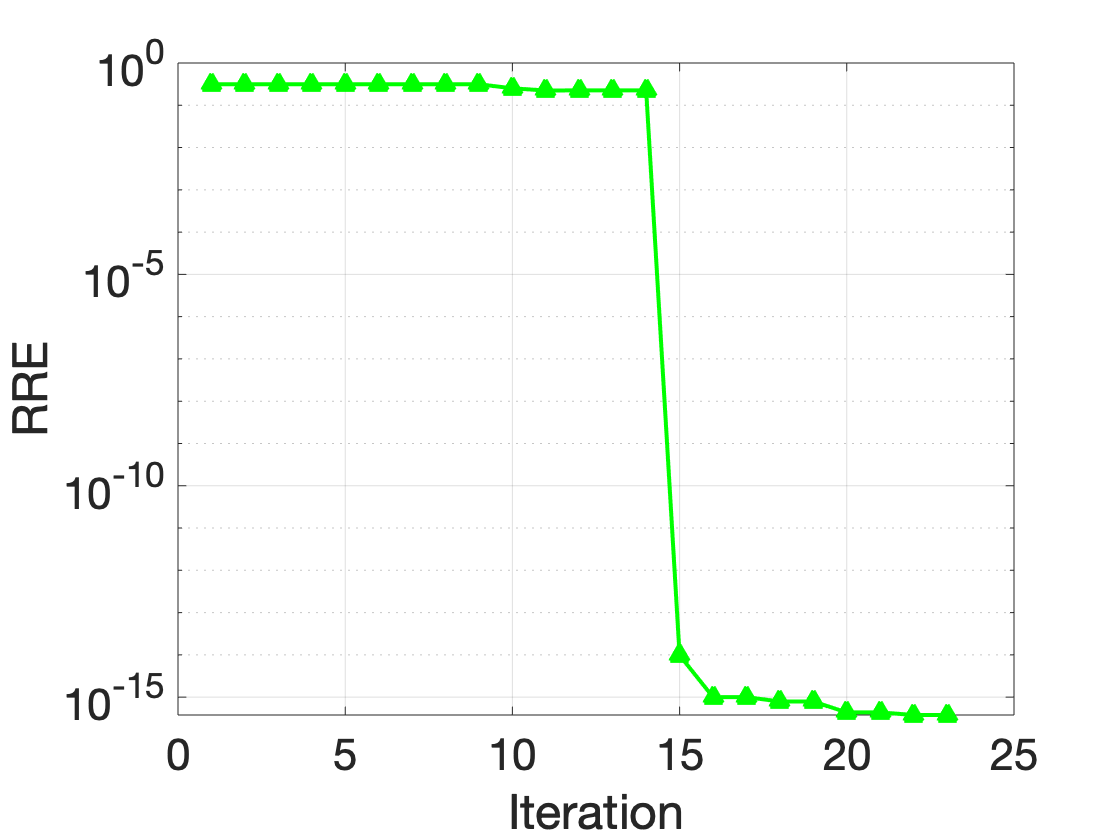} \\
\end{tabular} 
\caption{{Convergence behavior of the proposed method with the continuation strategy. \textbf{Top row:} results for the Gaussian matrix $A_1$. \textbf{Bottom row:} results for the oversampled DCT matrix $A_2$. \textbf{Left plots:} evolution of the parameter $\mu$ and the objective function value versus iterations. \textbf{Right plots:} evolution of the RRE versus iterations.}}
\label{muvsobjconstrained}
\end{figure}

\subsection{{Evaluation of exponential decay sparse signal recovery} }
\label{exp_decay}

{
We assess the capability of various methods to accurately discriminate significant non-zero entries from a broad dynamic range and exponential decay signal, including $L_1$ \cite{glowinski1975approximation}, $TL_1$ \cite{zhang2018minimization}, $L_0$ \cite{blumensath2009iterative}, $MCP$ \cite{sun2018sparse}, $L_p$ \cite{lai2013improved}, $L_1\mbox{-}L_2$ \cite{yin2015minimization}, $L_1/L_2$ \cite{tao2022minimization}, and the proposed $TH$. The reconstruction experiments are conducted in a noise-free setting using two different sensing matrices: $A_1 \in \mathbb{R}^{64\times 1024}$ with $r=0.5$, and $A_2 \in \mathbb{R}^{64\times 1024}$ with $F=5$. The ground truth signal $\bar{\vx} \in \mathbb{R}^{1024}$ is generated by
$\bar{x}_i = \frac{\sqrt{200}}{1+\exp\left(\frac{i-100}{20}\right)},\quad i=1:15:200$,
with the remaining entries set to zero. To eliminate any potential structural bias, the positions of the non-zero components are randomized.  The performance of various comparative methods is evaluated based on RRE, runtime, and iteration counts. 
}

{
As shown in Figure~\ref{decreasegt}, the results for the Gaussian matrix ($A_1$) are in the top panel, and for the DCT matrix ($A_2$) in the bottom panel. All absolute values are plotted on a logarithmic scale to enhance visual clarity (avoid ambiguities caused by the clustering of positive and negative fluctuations around zero).
}
{
The $TH$ method is the only approach that perfectly reconstructs the true signal in both test cases. For the Gaussian matrix, $TH$ achieves a remarkable RRE of $9.64\times10^{-16}$, and for the DCT matrix, an RRE of $6.64\times10^{-16}$. These error rates are at least $10^9$ times lower than those of the competing methods. Furthermore, the $TH$ method demonstrates exceptional computational efficiency, converging in just a few iterations (25 and 34, respectively). 
The $L_1/L_2$ and $L_1\mbox{-}L_2$ yield the next best RREs, with 5-10 times longer run times than that of the $TH$ method. Other methods like the $TL_1$ and the $MCP$, tend to introduce spurious non-zero entries that are not present in the ground truth, thus failing to recover the correct sparsity pattern. The $L_0$ method performs poorly in both scenarios, yielding high RREs of $9.14\times10^{-1}$ and $7.79\times10^{-1}$.
The true signal's non-zero entries span a wide dynamic range from $10^{-2}$ to $10^{2}$. Methods like $L_1$ and $L_p$ are only capable of recovering the large-magnitude components within the range of $10^{0}$ to $10^{2}$. The $L_1$ penalty, in particular, introduces a significant number of small-amplitude artifacts, while the $L_p$ method fails to produce a sufficiently sparse solution. This highlights the unique advantage of the $TH$ penalty in accurately recovering sparse signals with complex, decaying structures.
}

\begin{figure}
    \centering
    \begin{tabular}{c}
    \includegraphics[width=0.96\textwidth]{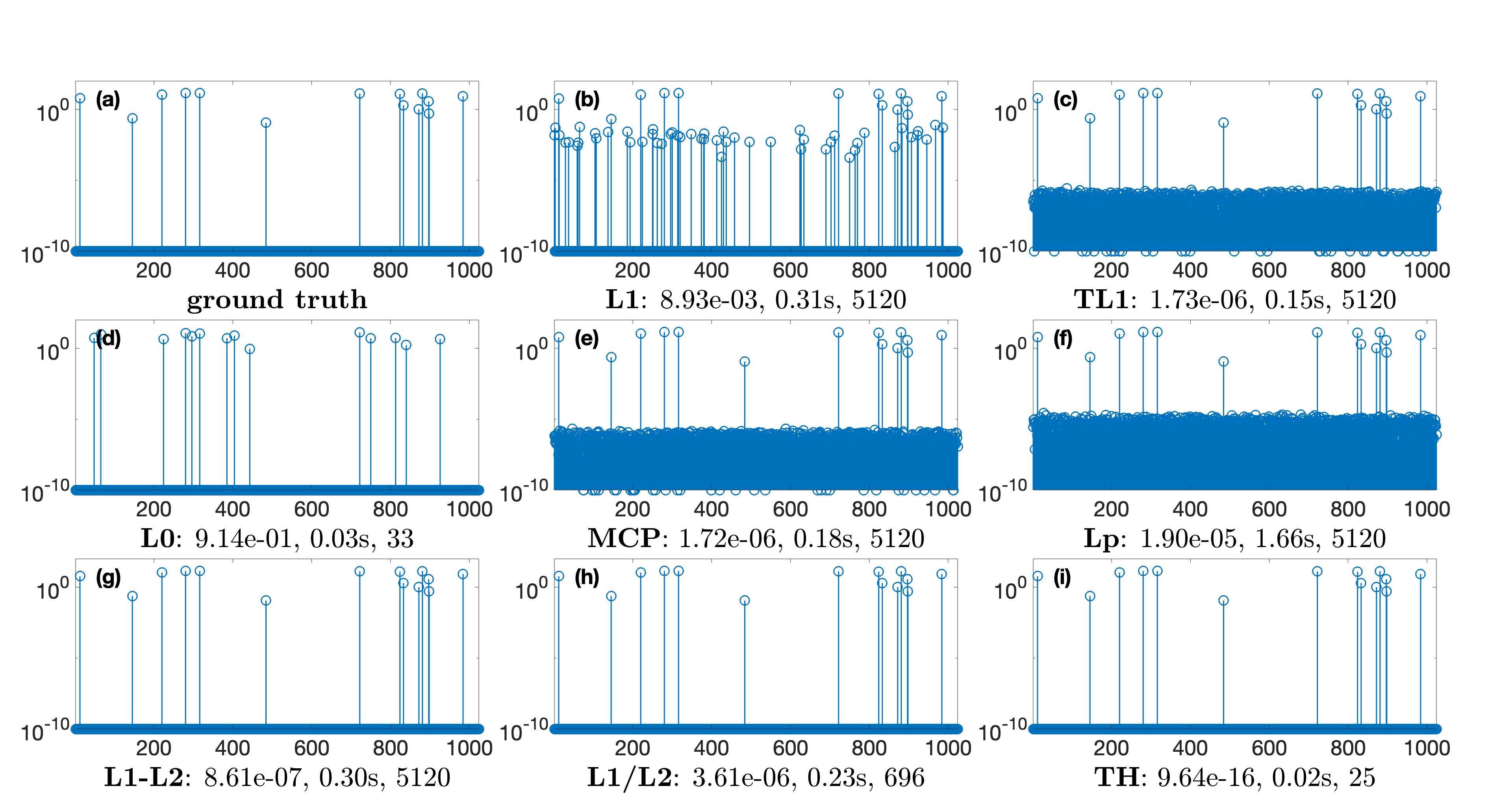} 
          \\[-0.8em]         
        \includegraphics[width=0.96\textwidth]{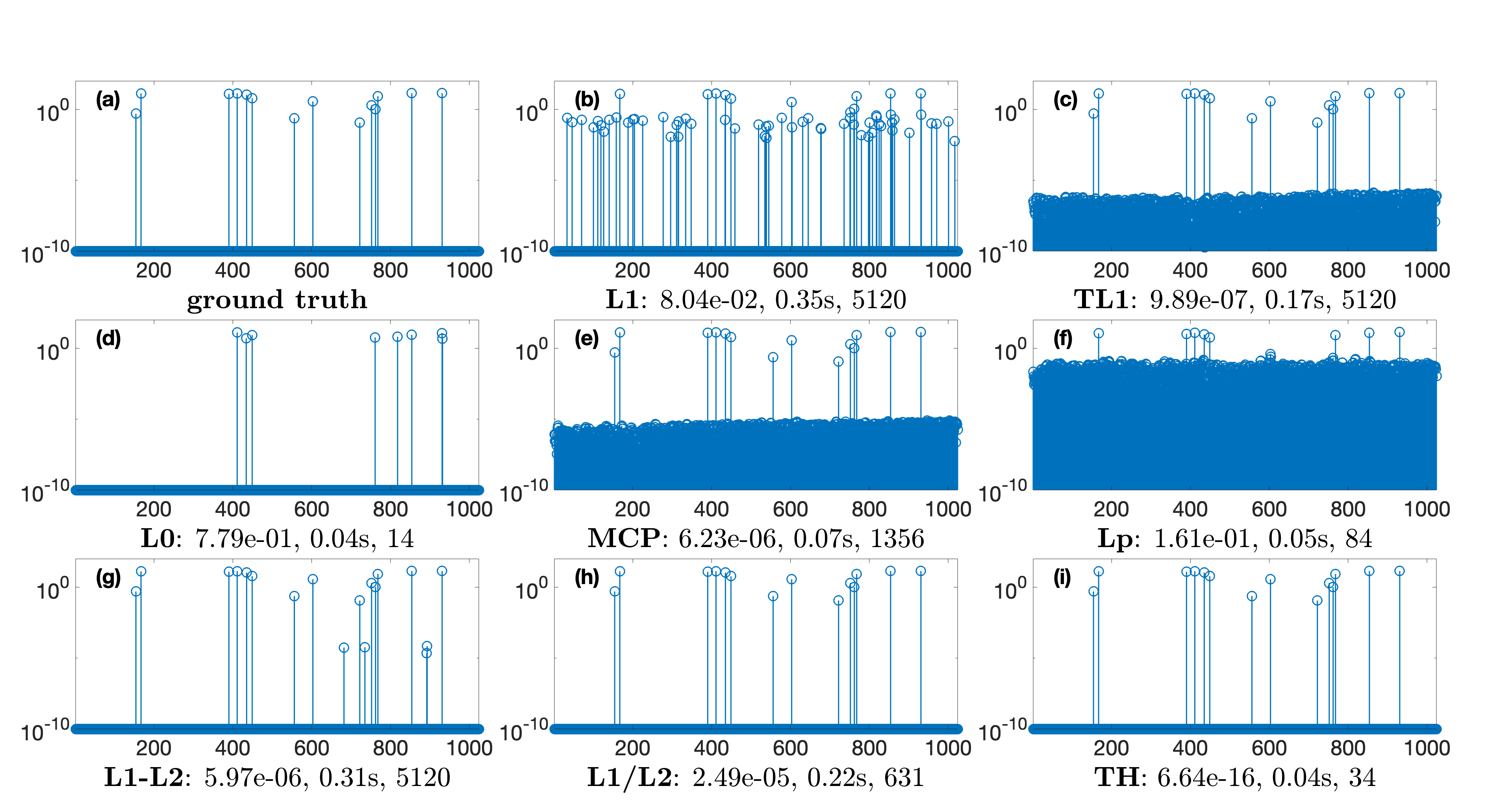}
   \end{tabular}    
   \caption{
   {Noise-free reconstruction results for a decaying signal. \textbf{Top panel}: Using a Gaussian sensing matrix $A_1 \in \mathbb{R}^{64\times 1024}$. \textbf{Bottom panel}: Using a oversampled DCT sensing matrix $A_2 \in \mathbb{R}^{64\times 1024}$. In each subplot title, the values represent the RRE, runtime in seconds, and iteration counts, respectively.}
 }
   \label{decreasegt}
\end{figure}

\subsection{Exact recovery in the noise-free setting}
For noise-free experiments, we test on two matrices: $A_1 \in \mathbb{R}^{64 \times 512}$ with $r = 0.8$, and $A_2 \in \mathbb{R}^{64 \times 2048}$ with $F = 10$.
The sparsity level $s$ varies from 2 to 30 in increments of 2.

\pgfplotstableread{
K L1 TL1 L0 MCP Lp L1-L2 L1/L2 TH
2  1.00 1.00 0.12 1.00 1.00 1.00 1.00 1.00
4  1.00 1.00 0.08 1.00 1.00 1.00 1.00 1.00
6  1.00 1.00 0.06 1.00 1.00 1.00 1.00 1.00
8  1.00 1.00 0.02 1.00 1.00 1.00 1.00 1.00
10 0.94 1.00 0.00 1.00 1.00 1.00 1.00 1.00
12 0.78 1.00 0.00 0.98 0.98 0.98 1.00 1.00
14 0.34 0.98 0.00 0.96 0.96 0.92 0.96 0.98
16 0.14 0.94 0.00 0.94 0.94 0.70 0.84 0.92
18 0.00 0.86 0.00 0.70 0.90 0.26 0.64 0.84
20 0.00 0.58 0.00 0.38 0.64 0.12 0.34 0.54
22 0.00 0.42 0.00 0.22 0.56 0.02 0.22 0.44
24 0.00 0.22 0.00 0.02 0.32 0.00 0.08 0.20
26 0.00 0.08 0.00 0.00 0.12 0.00 0.00 0.12
28 0.00 0.02 0.00 0.00 0.06 0.00 0.00 0.06
30 0.00 0.00 0.00 0.00 0.00 0.00 0.00 0.02
}\dataSRnew

\pgfplotstableread{
K L1 TL1 L0 MCP Lp L1-L2 L1/L2 TH
2  0.00 0.00 0.00 0.00 0.00 0.00 0.00 0.00
4  0.00 0.00 0.00 0.00 0.00 0.00 0.00 0.00
6  0.00 0.00 0.00 0.00 0.00 0.00 0.00 0.00
8  0.00 0.00 0.00 0.00 0.00 0.00 0.00 0.00
10 0.06 0.00 0.00 0.00 0.00 0.00 0.00 0.00
12 0.22 0.00 0.00 0.00 0.02 0.00 0.00 0.00
14 0.66 0.00 0.00 0.00 0.04 0.08 0.00 0.02
16 0.86 0.00 0.00 0.00 0.06 0.20 0.00 0.08
18 1.00 0.00 0.00 0.00 0.10 0.66 0.00 0.16
20 1.00 0.00 0.00 0.00 0.32 0.84 0.00 0.46
22 1.00 0.04 0.00 0.00 0.44 0.98 0.00 0.56
24 1.00 0.06 0.00 0.00 0.60 1.00 0.00 0.80
26 1.00 0.16 0.00 0.00 0.70 1.00 0.00 0.88
28 1.00 0.22 0.00 0.00 0.82 1.00 0.00 0.94
30 1.00 0.34 0.00 0.00 0.88 1.00 0.00 0.98
}\dataMFnew

\pgfplotstableread{
K L1 TL1 L0 MCP Lp L1-L2 L1/L2 TH
2  0.00 0.00 0.88 0.00 0.00 0.00 0.00 0.00
4  0.00 0.00 0.92 0.00 0.00 0.00 0.00 0.00
6  0.00 0.00 0.94 0.00 0.00 0.00 0.00 0.00
8  0.00 0.00 0.98 0.00 0.00 0.00 0.00 0.00
10 0.00 0.00 1.00 0.00 0.00 0.00 0.00 0.00
12 0.00 0.00 1.00 0.02 0.00 0.02 0.00 0.00
14 0.00 0.02 1.00 0.04 0.00 0.00 0.04 0.00
16 0.00 0.06 1.00 0.06 0.00 0.10 0.16 0.00
18 0.00 0.14 1.00 0.30 0.00 0.08 0.36 0.00
20 0.00 0.42 1.00 0.62 0.04 0.04 0.66 0.00
22 0.00 0.54 1.00 0.78 0.00 0.00 0.78 0.00
24 0.00 0.72 1.00 0.98 0.08 0.00 0.92 0.00
26 0.00 0.76 1.00 1.00 0.18 0.00 1.00 0.00
28 0.00 0.76 1.00 1.00 0.12 0.00 1.00 0.00
30 0.00 0.66 1.00 1.00 0.12 0.00 1.00 0.00
}\dataAFnew


\pgfplotstableread{

K L1 TL1 L0 MCP Lp L1-L2 L1/L2 TH

2 1.0000 1.0000 0.8200 1.0000 0.9800 1.0000 1.0000 1.0000

4 1.0000 1.0000 0.2600 1.0000 0.8400 1.0000 1.0000 1.0000

6 1.0000 1.0000 0.1000 1.0000 0.6400 1.0000 1.0000 1.0000

8 0.9800 1.0000 0.0800 1.0000 0.3800 1.0000 1.0000 1.0000

10 0.6400 0.8400 0.0000 0.9600 0.1400 0.9600 0.8200 0.8200

12 0.2600 0.6800 0.0000 0.7000 0.0400 0.7200 0.5800 0.5600

14 0.0800 0.4000 0.0000 0.4200 0.0000 0.4000 0.3400 0.3000

16 0.0000 0.0400 0.0000 0.0800 0.0200 0.0600 0.0400 0.0800

18 0.0000 0.0000 0.0000 0.0000 0.0000 0.0000 0.0000 0.0200

20 0.0000 0.0200 0.0000 0.0000 0.0000 0.0000 0.0000 0.0200

22 0.0000 0.0000 0.0000 0.0000 0.0000 0.0000 0.0000 0.0200

24 0.0000 0.0000 0.0000 0.0000 0.0000 0.0000 0.0000 0.0000

26 0.0000 0.0000 0.0000 0.0000 0.0000 0.0000 0.0000 0.0000

28 0.0000 0.0000 0.0000 0.0000 0.0000 0.0000 0.0000 0.0000

30 0.0000 0.0000 0.0000 0.0000 0.0000 0.0000 0.0000 0.0000

}\dataSR

\pgfplotstableread{

K L1 TL1 L0 MCP Lp L1-L2 L1/L2 TH

2 0.0000 0.0000 0.0000 0.0000 0.0000 0.0000 0.0000 0.0000

4 0.0000 0.0000 0.0000 0.0000 0.0000 0.0000 0.0000 0.0000

6 0.0000 0.0000 0.0000 0.0000 0.0000 0.0000 0.0000 0.0000

8 0.0200 0.0000 0.0000 0.0000 0.0000 0.0000 0.0000 0.0000

10 0.3600 0.0000 0.0000 0.0000 0.0000 0.0000 0.0000 0.1800

12 0.7400 0.0000 0.0000 0.0000 0.0000 0.0600 0.0200 0.4400

14 0.9200 0.0200 0.0000 0.0000 0.0000 0.2200 0.0000 0.7000

16 1.0000 0.0800 0.0000 0.0000 0.0200 0.6200 0.0000 0.9200

18 1.0000 0.1400 0.0000 0.0000 0.0000 0.7800 0.0000 0.9800

20 1.0000 0.3400 0.0000 0.0000 0.0000 0.9600 0.0000 0.9800

22 1.0000 0.5000 0.0000 0.0000 0.0800 1.0000 0.0000 0.9800

24 1.0000 0.8800 0.0000 0.0000 0.0800 1.0000 0.0200 1.0000

26 1.0000 0.9000 0.0000 0.0000 0.3000 1.0000 0.0000 1.0000

28 1.0000 1.0000 0.0000 0.0000 0.5200 1.0000 0.0000 1.0000

30 1.0000 1.0000 0.0000 0.0000 0.8000 1.0000 0.0000 1.0000

}\dataMF

\pgfplotstableread{

K L1 TL1 L0 MCP Lp L1-L2 L1/L2 TH

2 0.0000 0.0000 0.1800 0.0000 0.0200 0.0000 0.0000 0.0000

4 0.0000 0.0000 0.7400 0.0000 0.1600 0.0000 0.0000 0.0000

6 0.0000 0.0000 0.9000 0.0000 0.3600 0.0000 0.0000 0.0000

8 0.0000 0.0000 0.9200 0.0000 0.6200 0.0000 0.0000 0.0000

10 0.0000 0.1600 1.0000 0.0400 0.8600 0.0400 0.1800 0.0000

12 0.0000 0.3200 1.0000 0.3000 0.9600 0.2200 0.4000 0.0000

14 0.0000 0.5800 1.0000 0.5800 1.0000 0.3800 0.6600 0.0000

16 0.0000 0.8800 1.0000 0.9200 0.9600 0.3200 0.9600 0.0000

18 0.0000 0.8600 1.0000 1.0000 1.0000 0.2200 1.0000 0.0000

20 0.0000 0.6400 1.0000 1.0000 1.0000 0.0400 1.0000 0.0000

22 0.0000 0.5000 1.0000 1.0000 0.9200 0.0000 1.0000 0.0000

24 0.0000 0.1200 1.0000 1.0000 0.9200 0.0000 0.9800 0.0000

26 0.0000 0.1000 1.0000 1.0000 0.7000 0.0000 1.0000 0.0000

28 0.0000 0.0000 1.0000 1.0000 0.4800 0.0000 1.0000 0.0000

30 0.0000 0.0000 1.0000 1.0000 0.2000 0.0000 1.0000 0.0000

}\dataAF

\begin{figure}
\label{noiseless}
\centering
\begin{tikzpicture}
\begin{groupplot}[
group style={group size=3 by 2,horizontal sep=1cm,vertical sep=1.2cm,},
width=4.8cm,height=4.25cm,
xlabel style={font=\fontsize{6}{6}\selectfont},
ylabel style={font=\fontsize{6}{6}\selectfont},
xticklabel style={font=\fontsize{6}{6}\selectfont},
yticklabel style={font=\fontsize{6}{6}\selectfont},
grid=none,
minor tick num=1,
legend style={draw=none,fill=none,text depth=0pt,text height=0.8ex},
legend columns=-1,
]
\nextgroupplot[xlabel={Sparsity},xmin=2, xmax=30, xtick={2,9,16,23,30},] 
\addplot[blue, mark=*, mark options={fill=blue}, mark repeat=3] table[x=K,y=L1] {\dataSRnew};
\addplot[red, mark=diamond*, mark options={fill=red}, mark repeat=3] table[x=K,y=TL1] {\dataSRnew};
\addplot[magenta, mark=+, mark repeat=3] table[x=K,y=L0] {\dataSRnew};
\addplot[cyan, dashed, mark=triangle*, mark options={fill=cyan, rotate=90}, mark repeat=3] table[x=K,y=MCP] {\dataSRnew};
\addplot[magenta, dashed, mark=square*, mark options={fill=magenta}, mark repeat=3] table[x=K,y=Lp] {\dataSRnew};
\addplot[green, mark=x, mark repeat=3] table[x=K,y=L1-L2] {\dataSRnew};
\addplot[orange, mark=pentagon*, mark options={fill=orange}, mark repeat=3] table[x=K,y=L1/L2] {\dataSRnew};
\addplot[black, mark=triangle*, mark options={fill=black}, mark repeat=3] table[x=K,y=TH] {\dataSRnew};

\nextgroupplot[xlabel={Sparsity}, xmin=2, xmax=30, xtick={2,9,16,23,30},] 
\addplot[blue, mark=*, mark options={fill=blue}, mark repeat=3] table[x=K,y=L1] {\dataMFnew};
\addplot[red, mark=diamond*, mark options={fill=red}, mark repeat=3] table[x=K,y=TL1] {\dataMFnew};
\addplot[magenta, mark=+, mark repeat=3] table[x=K,y=L0] {\dataMFnew};
\addplot[cyan, dashed, mark=triangle*, mark options={fill=cyan, rotate=90}, mark repeat=3] table[x=K,y=MCP] {\dataMFnew};
\addplot[magenta, dashed, mark=square*, mark options={fill=magenta}, mark repeat=3] table[x=K,y=Lp] {\dataMFnew};
\addplot[green, mark=x, mark repeat=3] table[x=K,y=L1-L2] {\dataMFnew};
\addplot[orange, mark=pentagon*, mark options={fill=orange}, mark repeat=3] table[x=K,y=L1/L2] {\dataMFnew};
\addplot[black, mark=triangle*, mark options={fill=black}, mark repeat=3] table[x=K,y=TH] {\dataMFnew};

\nextgroupplot[xlabel={Sparsity}, xmin=2, xmax=30, xtick={2,9,16,23,30}, ] 
\addplot[blue, mark=*, mark options={fill=blue}, mark repeat=3] table[x=K,y=L1] {\dataAFnew};
\addplot[red, mark=diamond*, mark options={fill=red}, mark repeat=3] table[x=K,y=TL1] {\dataAFnew};
\addplot[magenta, mark=+, mark repeat=2] table[x=K,y=L0] {\dataAFnew};
\addplot[cyan, dashed, mark=triangle*, mark options={fill=cyan, rotate=90}, mark repeat=3] table[x=K,y=MCP] {\dataAFnew};
\addplot[magenta, dashed, mark=square*, mark options={fill=magenta}, mark repeat=3] table[x=K,y=Lp] {\dataAFnew};
\addplot[green, mark=x, mark repeat=3] table[x=K,y=L1-L2] {\dataAFnew};
\addplot[orange, mark=pentagon*, mark options={fill=orange}, mark repeat=3] table[x=K,y=L1/L2] {\dataAFnew};
\addplot[black, mark=triangle*, mark options={fill=black}, mark repeat=4] table[x=K,y=TH] {\dataAFnew};

\nextgroupplot[xlabel={Sparsity}, xmin=2, xmax=30, xtick={2,9,16,23,30}, ] 
\addplot[blue, mark=*, mark options={fill=blue}, mark repeat=3] table[x=K,y=L1] {\dataSR};
\addplot[red, mark=diamond*, mark options={fill=red}, mark repeat=3] table[x=K,y=TL1] {\dataSR};
\addplot[magenta, mark=+, mark repeat=3] table[x=K,y=L0] {\dataSR};
\addplot[cyan, dashed, mark=triangle*, mark options={fill=cyan, rotate=90}, mark repeat=3] table[x=K,y=MCP] {\dataSR};
\addplot[magenta, dashed, mark=square*, mark options={fill=magenta}, mark repeat=3] table[x=K,y=Lp] {\dataSR};
\addplot[green, mark=x, mark repeat=3] table[x=K,y=L1-L2] {\dataSR};
\addplot[orange, mark=pentagon*, mark options={fill=orange}, mark repeat=3] table[x=K,y=L1/L2] {\dataSR};
\addplot[black, mark=triangle*, mark options={fill=black}, mark repeat=3] table[x=K,y=TH] {\dataSR};

\nextgroupplot[xlabel={Sparsity}, xmin=2, xmax=30, xtick={2,9,16,23,30}, ] 
\addplot[blue, mark=*, mark options={fill=blue}, mark repeat=3] table[x=K,y=L1] {\dataMF};
\addplot[red, mark=diamond*, mark options={fill=red}, mark repeat=3] table[x=K,y=TL1] {\dataMF};
\addplot[magenta, mark=+, mark repeat=3] table[x=K,y=L0] {\dataMF};
\addplot[cyan, dashed, mark=triangle*, mark options={fill=cyan, rotate=90}, mark repeat=3] table[x=K,y=MCP] {\dataMF};
\addplot[magenta, dashed, mark=square*, mark options={fill=magenta}, mark repeat=3] table[x=K,y=Lp] {\dataMF};
\addplot[green, mark=x, mark repeat=3] table[x=K,y=L1-L2] {\dataMF};
\addplot[orange, mark=pentagon*, mark options={fill=orange}, mark repeat=3] table[x=K,y=L1/L2] {\dataMF};
\addplot[black, mark=triangle*, mark options={fill=black}, mark repeat=3] table[x=K,y=TH] {\dataMF};

\nextgroupplot[xlabel={Sparsity}, xmin=2, xmax=30, xtick={2,9,16,23,30}, ] 
\addplot[blue, mark=*, mark options={fill=blue}, mark repeat=3] table[x=K,y=L1] {\dataAF};
\addplot[red, mark=diamond*, mark options={fill=red}, mark repeat=3] table[x=K,y=TL1] {\dataAF};
\addplot[magenta, mark=+, mark repeat=3] table[x=K,y=L0] {\dataAF};
\addplot[cyan, dashed, mark=triangle*, mark options={fill=cyan, rotate=90}, mark repeat=3] table[x=K,y=MCP] {\dataAF};
\addplot[magenta, dashed, mark=square*, mark options={fill=magenta}, mark repeat=3] table[x=K,y=Lp] {\dataAF};
\addplot[green, mark=x, mark repeat=3] table[x=K,y=L1-L2] {\dataAF};
\addplot[orange, mark=pentagon*, mark options={fill=orange}, mark repeat=3] table[x=K,y=L1/L2] {\dataAF};
\addplot[black, mark=triangle*, mark options={fill=black}, mark repeat=4] table[x=K,y=TH] {\dataAF};

\end{groupplot}

\node[draw, font=\fontsize{6}{6}\selectfont, inner ysep=0.05cm] at ([yshift=.6cm]group c2r1.north) {
\begin{tikzpicture}
\matrix [column sep=0.12cm, ampersand replacement=\&] {
\node {\tikz{\draw[blue, line width=0.8pt] (0,0) -- (0.4,0); \node[blue] at (0.2,0) {\pgfuseplotmark{*}};}}; \& \node {$L_1$}; \&
\node {\tikz{\draw[red, line width=0.8pt] (0,0) -- (0.4,0); \node[red] at (0.2,0) {\pgfuseplotmark{diamond*}};}}; \& \node {$TL_1$}; \&
\node {\tikz{\draw[magenta, line width=0.8pt] (0,0) -- (0.4,0); \node[magenta] at (0.2,0) {\pgfuseplotmark{+}};}}; \& \node {$L_0$}; \&
\node {\tikz{\draw[cyan, dashed, line width=0.8pt] (0,0) -- (0.4,0); \node[cyan, rotate=90] at (0.15,0) {\pgfuseplotmark{triangle*}};}}; \& \node {$MCP$}; \&
\node {\tikz{\draw[magenta, dashed, line width=0.8pt] (0,0) -- (0.4,0); \node[magenta] at (0.15,0) {\pgfuseplotmark{square*}};}}; \& \node {$L_p$}; \&
\node {\tikz{\draw[green, line width=0.8pt] (0,0) -- (0.4,0); \node[green] at (0.2,0) {\pgfuseplotmark{x}};}}; \& \node {$L_1 - L_2$}; \&
\node {\tikz{\draw[orange, line width=0.8pt] (0,0) -- (0.4,0); \node[orange] at (0.2,0) {\pgfuseplotmark{pentagon*}};}}; \& \node {$L_1/L_2$}; \&
\node {\tikz{\draw[black, line width=0.8pt] (0,0) -- (0.4,0); \node[black] at (0.2,0) {\pgfuseplotmark{triangle*}};}}; \& \node {$TH$}; \\
};
\end{tikzpicture}
};
\end{tikzpicture}
\vspace{-1em}
\caption{{Comparison of sparse reconstruction performance in the noise-free case using matrices $A_1$ (with $r=0.8$; top row) and $A_2$ (with $F=10$; bottom row). Metrics include success rates (\textit{left column}), model failures (\textit{middle column}), and algorithm failures (\textit{right column}), evaluated for the following regularizers: $L_1$, $TL_1$, $L_0$, $MCP$, $L_p$, $L_1\mbox{-}L_2$, $L_1/L_2$, and $TH$.} }
\end{figure}

The recovery performance is quantified by the success rate, defined as the percentage of successful recoveries over 50 independent trials. A trial is successful if the RRE is less than $10^{-3}$.
Failures (RRE $>10^{-3}$) are further classified as:
\begin{enumerate}[1),leftmargin=*]
\item model failures: occur when $\Phi(\overline{\vx}) > \Phi(\wvx)$, which contradicts the assumption of $\overline{\vx}$ being the global minimizer;
\item algorithm failures: occur when $\Phi(\overline{\vx}) < \Phi(\wvx)$, indicating that the algorithm did not converge to a satisfactory minimizer.
\end{enumerate}

Figure~\ref{noiseless} displays the success rates, model failures, and algorithm failures for various models.
For the Gaussian matrix ($A_1$, top panel), while the $L_p$ model achieves the highest success rate, both the proposed $TH$ and the $TL_1$ models exhibit highly competitive performance. {The $MCP$ and $L_1/L_2$ models perform slightly worse success rates, accompanied by higher algorithm failures.} 
For the DCT matrix ($A_2$, bottom panel), although the $L_1\mbox{-}L_2$ and $MCP$ models lead in terms of success rate, the $TH$ method remains consistently robust, ranking closely with the $TL_1$ and $L_1/L_2$ models. Notably, the performance of the $L_p$ model declines significantly for $A_2$, {whereas the $L_1/L_2$ model shows improved performance.} The $TH$ method maintains stable behavior across different matrix types and sparsity levels. For both matrix types, the $L_1$ model consistently underperforms, and the $L_0$ model nearly always fails.

Figure~\ref{noiselesstimetable} reports the average computation times (in seconds) for the experiments depicted in Figure \ref{noiseless}, with each value averaged over 50 independent trials. Despite the fact that the $L_0$ model records the shortest computation times, its reconstructions are almost invariably unsuccessful. 
In contrast, the $TH$ method exhibits the second fastest performance, particularly at low sparsity levels (with computation times nearly an order of magnitude lower than most competitors). 
This computational efficiency, combined with its robust and consistent recovery performance across a variety of experimental settings, underscores the practical value of the $TH$ approach.
{The $L_1/L_2$ and $MCP$ models are the slowest, especially at high sparsity levels for $A_2$.} The computation times for $TH$, as well as for the $L_p$, $L_1\mbox{-}L_2$, $TL_1$, $L_1/L_2$, and $L_0$ models, increase with sparsity due to the higher number of iterations needed to satisfy the stopping criterion \eqref{stopcondition}.

\pgfplotstableread{
Sparsity L1 TL1 L1L2 L0 MCP Lp TH L1divL2
2 0.0423 0.0414 0.0416 0.0017 0.4397 0.0028 0.0045 0.0335
6 0.0457 0.0432 0.0438 0.0017 0.4323 0.0034 0.0036 0.4115
10 0.0499 0.0459 0.0480 0.0024 0.4328 0.0043 0.0037 0.4160
14 0.0523 0.0872 0.0968 0.0028 0.4323 0.0098 0.0056 0.4117
18 0.0519 0.1972 0.2510 0.0035 0.4182 0.0490 0.0109 0.4126
22 0.0511 0.4517 0.3284 0.0042 0.4218 0.1403 0.0179 0.4050
26 0.0516 0.5647 0.3230 0.0055 0.4255 0.2790 0.0220 0.4094
30 0.0524 0.5974 0.3243 0.0060 0.4246 0.3104 0.0234 0.4092
}\AoneData

\pgfplotstableread{
Sparsity L1 TL1 L1L2 L0 MCP Lp TH L1divL2
2  0.2047  0.2049  0.2111  0.0021  0.2673  0.0123  0.0108  0.3661
6  0.2304  0.2157  0.2220  0.0069  0.3232  0.0324  0.0024  1.0994
10 0.2663  0.4825  0.3766  0.0117  0.5547  0.0565  0.0179  0.9861
14 0.2982  1.9259  1.7962  0.0166  3.0744  0.1026  0.0471  3.8079
18 0.2924  3.0353  2.5661  0.0247  4.4278  0.1172  0.0581  3.8690
22 0.2930  3.0327  2.6251  0.0313  4.3292  0.1362  0.0583  3.7572
26 0.2960  2.5793  2.4594  0.0303  4.2020  0.1339  0.0572  3.5647
30 0.3242  3.0694  2.8373  0.0500  5.1092  0.1505  0.0667  4.2151
}\AtwoData

\begin{figure}
\centering
\begin{tikzpicture}
\begin{groupplot}[
    group style={
        group size=2 by 1,  
        horizontal sep=1.6cm,  
        vertical sep=1cm,
        group name=myplots,   
    },
    width=6.6cm,
    height=4.5cm,
    xlabel={Sparsity},
    xlabel style={font=\fontsize{6}{6}\selectfont},
    ylabel style={font=\fontsize{6}{6}\selectfont},
    xticklabel style={font=\fontsize{6}{6}\selectfont},
    yticklabel style={font=\fontsize{6}{6}\selectfont},
    ymode=log,
    xmin=2, xmax=30,
    xtick={2,6,10,14,18,22,26,30},
    grid=none,
    minor tick num=1,
]

\nextgroupplot
\addplot[blue, mark=*, mark options={fill=blue}, mark repeat=2] table[x=Sparsity,y=L1] {\AoneData};
\addplot[red, mark=diamond*, mark options={fill=red}, mark repeat=2] table[x=Sparsity,y=TL1] {\AoneData};
\addplot[green, mark=x, mark repeat=2] table[x=Sparsity,y=L1L2] {\AoneData};
\addplot[magenta, mark=+, mark repeat=2] table[x=Sparsity,y=L0] {\AoneData};
\addplot[cyan, dashed, mark=triangle*, mark options={fill=cyan, rotate=90}, mark repeat=2] table[x=Sparsity,y=MCP] {\AoneData};
\addplot[magenta, dashed, mark=square*, mark options={fill=magenta}, mark repeat=2] table[x=Sparsity,y=Lp] {\AoneData};
\addplot[black, mark=triangle*, mark options={fill=black}, mark repeat=2] table[x=Sparsity,y=TH] {\AoneData};
\addplot[orange, mark=pentagon*, mark options={fill=orange}, mark repeat=2] table[x=Sparsity,y=L1divL2] {\AoneData};

\nextgroupplot
\addplot[blue, mark=*, mark options={fill=blue}, mark repeat=2] table[x=Sparsity,y=L1] {\AtwoData};
\addplot[red, mark=square*, mark repeat=2] table[x=Sparsity,y=TL1] {\AtwoData};
\addplot[green, mark=x, mark repeat=2] table[x=Sparsity,y=L1L2] {\AtwoData};
\addplot[magenta, mark=+, mark repeat=2] table[x=Sparsity,y=L0] {\AtwoData};
\addplot[cyan, dashed, mark=triangle*, mark options={fill=cyan, rotate=90}, mark repeat=2] table[x=Sparsity,y=MCP] {\AtwoData};
\addplot[magenta, dashed, mark=square*, mark options={fill=magenta}, mark repeat=2] table[x=Sparsity,y=Lp] {\AtwoData};
\addplot[orange, mark=pentagon*, mark options={fill=orange}] table[x=Sparsity,y=L1divL2] {\AtwoData};
\addplot[black, mark=triangle*, mark options={fill=black}, mark repeat=2] table[x=Sparsity,y=TH] {\AtwoData};

\end{groupplot}

\node[draw, font=\fontsize{6}{6}\selectfont, inner ysep=0.05cm] at ([yshift=0.8cm]group c2r1.north) {
\begin{tikzpicture}
\matrix [column sep=0.05cm, ampersand replacement=\&] {
\node {\tikz{\draw[blue, line width=0.8pt] (0,0) -- (0.4,0); \node[blue] at (0.2,0) {\pgfuseplotmark{*}};}}; \& \node {$L_1$}; \&
\node {\tikz{\draw[red, line width=0.8pt] (0,0) -- (0.4,0); \node[red] at (0.2,0) {\pgfuseplotmark{diamond*}};}}; \& \node {$TL_1$}; \&
\node {\tikz{\draw[magenta, line width=0.8pt] (0,0) -- (0.4,0); \node[magenta] at (0.2,0) {\pgfuseplotmark{+}};}}; \& \node {$L_0$}; \&
\node {\tikz{\draw[cyan, dashed, line width=0.8pt] (0,0) -- (0.4,0); \node[cyan, rotate=90] at (0.15,0) {\pgfuseplotmark{triangle*}};}}; \& \node {$MCP$}; \&
\node {\tikz{\draw[magenta, dashed, line width=0.8pt] (0,0) -- (0.4,0); \node[magenta] at (0.15,0) {\pgfuseplotmark{square*}};}}; \& \node {$L_p$}; \&
\node {\tikz{\draw[green, line width=0.8pt] (0,0) -- (0.4,0); \node[green] at (0.2,0) {\pgfuseplotmark{x}};}}; \& \node {$L_1 - L_2$}; \&
\node {\tikz{\draw[orange, line width=0.8pt] (0,0) -- (0.4,0); \node[orange] at (0.2,0) {\pgfuseplotmark{pentagon*}};}}; \& \node {$L_1/L_2$}; \&
\node {\tikz{\draw[black, line width=0.8pt] (0,0) -- (0.4,0); \node[black] at (0.2,0) {\pgfuseplotmark{triangle*}};}}; \& \node {$TH$}; \\
};
\end{tikzpicture}
};
\end{tikzpicture}
\caption{{Computation time comparison for sparse vector recovery under noise-free conditions using matrices $A_1$ ($r = 0.8$; left) and $A_2$ ($F = 10$; right). Time values (in seconds) are averaged over 50 independent trials.}}
\label{noiselesstimetable}
\end{figure}

\subsection{Robust recovery under noisy observations}\label{snrsubsection}
To assess the robustness and accuracy of various methods in the presence of noise, we perform experiments on both Gaussian and oversampled DCT matrices over a range of signal-to-noise ratio (SNR\footnote{SNR is defined as $SNR(\vb, A\bar\vx) = 10\log_{10}\tfrac{\|\vb - A\bar\vx\|_2^2}{\|A\bar\vx\|_2^2}$}) levels. A higher SNR corresponds to a lower noise level. For each SNR value, the RRE is averaged over 50 independent trials. 

According to Figure~\ref{vary_SNR},
for the Gaussian matrix $A_1 \in \mathbb{R}^{64 \times 512}$ with $r = 0.5$ and sparsity level $s = 12$, the $TH$ method consistently achieves the lowest RRE across SNR levels. The $L_p$ and $TL_1$ models show reduced performance compared to their noise-free scenario, while trends for $MCP$, $L_1/L_2$, $L_1\mbox{-}L_2$, $L_1$, and $L_0$ are consistent with the noise-free case.
For the DCT matrix $A_2 \in \mathbb{R}^{64 \times 1024}$ ($F = 5$, $s = 8$), the $TH$ method consistently outperforms all other methods across various noise levels. 
{
The $L_p$ method achieves suboptimal performance at low SNR for $A_1$. $L_1/L_2$ delivers moderate reconstruction on $A_1$ but demonstrates superior high-SNR scalability on $A_2$. 
$L_0$ shows consistently poor accuracy on $A_1$. 
Both $TL_1$ and $MCP$ exhibit analogous error progression, with $TL_1$ maintaining consistent advantage over $MCP$ at elevated SNR levels.
} 
A similar trend is observed for the $L_p$ method. Additionally, the performance of the $MCP$, $L_1-L_2$, and $TL_1$ methods degrades compared to those in noiseless scenarios.

\pgfplotstableread{
SNR   L1     TL1    L1L2   L0     MCP    Lp     TH     L1divL2
18    0.1771 0.0609 0.0954 0.5452 0.0699 0.0415 0.0333 0.0726
21    0.1466 0.0426 0.0656 0.5308 0.0485 0.0273 0.0212 0.0486
24    0.1211 0.0300 0.0464 0.5074 0.0346 0.0187 0.0146 0.0335
27    0.1023 0.0213 0.0331 0.4992 0.0245 0.0130 0.0100 0.0243
30    0.0870 0.0151 0.0238 0.5200 0.0174 0.0090 0.0068 0.0168
33    0.0756 0.0106 0.0170 0.4959 0.0124 0.0062 0.0045 0.0119
36    0.0672 0.0075 0.0121 0.4993 0.0088 0.0044 0.0032 0.0085
39    0.0606 0.0053 0.0086 0.4981 0.0062 0.0031 0.0022 0.0060
42    0.0562 0.0038 0.0061 0.5094 0.0044 0.0021 0.0015 0.0042
45    0.0529 0.0027 0.0043 0.5059 0.0031 0.0015 0.0011 0.0030
48    0.0506 0.0019 0.0030 0.5020 0.0022 0.0011 0.0007 0.0021
51    0.0489 0.0013 0.0022 0.4968 0.0016 0.0007 0.0005 0.0015
54    0.0477 0.0009 0.0015 0.4968 0.0011 0.0005 0.0004 0.0011
57    0.0470 0.0007 0.0011 0.4915 0.0008 0.0004 0.0003 0.0007
60    0.0462 0.0005 0.0008 0.5047 0.0006 0.0003 0.0002 0.0005
}\leftData

\pgfplotstableread{
SNR   L1     TL1    L1L2   L0     MCP    Lp     TH     L1divL2
30    1.6593 3.6506 2.1794 0.4138 3.7203 4.0526 0.4268 2.0776
33    1.0230 1.0651 1.0863 0.3856 1.7370 2.8236 0.3586 1.4759
36    0.7943 0.7944 0.7950 0.3478 0.8358 1.8813 0.2669 1.1165
39    0.6192 0.5985 0.5994 0.3084 0.6257 0.4469 0.2157 0.8178
42    0.4783 0.4284 0.4390 0.2744 0.4554 0.3522 0.1879 0.6014
45    0.3573 0.2819 0.2958 0.2721 0.3117 0.2351 0.0799 0.4120
48    0.2679 0.1928 0.2117 0.2926 0.2262 0.1756 0.0472 0.4394
51    0.1985 0.1242 0.1420 0.2798 0.1535 0.1373 0.0302 0.0207
54    0.1456 0.0766 0.0932 0.2592 0.1048 0.0907 0.0180 0.0141
57    0.1047 0.0463 0.0635 0.2667 0.0725 0.0640 0.0169 0.0096
60    0.0709 0.0270 0.0414 0.2652 0.0464 0.0502 0.0038 0.0066
63    0.0487 0.0145 0.0272 0.2744 0.0305 0.0460 0.0022 0.0045
66    0.0326 0.0088 0.0160 0.2741 0.0189 0.0444 0.0014 0.0031
69    0.0202 0.0054 0.0090 0.2739 0.0106 0.0436 0.0011 0.0023
72    0.0126 0.0034 0.0054 0.2831 0.0063 0.0430 0.0007 0.0019
}\rightData
\begin{figure}[htb]
\centering
\begin{tikzpicture}
\begin{groupplot}[
    group style={
        group size=2 by 1,        
        horizontal sep=1.6cm,     
        vertical sep=1cm,
        group name=myplots,       
    },
    width=6.2cm,                  
    height=4.5cm,                 
    xlabel={SNR (dB)},
    ylabel={RRE},
    xlabel style={font=\fontsize{6}{6}\selectfont},
    ylabel style={font=\fontsize{6}{6}\selectfont},
    xticklabel style={font=\fontsize{6}{6}\selectfont},
    yticklabel style={font=\fontsize{6}{6}\selectfont},
    ymode=log,
    grid=none,
    minor tick num=1,
]

\nextgroupplot[
    xmin=18, xmax=60,
    xtick={18,27,36,45,54},
]
\addplot[blue, mark=*, mark options={fill=blue}, mark repeat=2] table[x=SNR,y=L1] {\leftData};
\addplot[red, mark=diamond*, mark options={fill=red}, mark repeat=2] table[x=SNR,y=TL1] {\leftData};
\addplot[green, mark=x, mark repeat=2] table[x=SNR,y=L1L2] {\leftData};
\addplot[magenta, mark=+, mark repeat=2] table[x=SNR,y=L0] {\leftData};
\addplot[cyan, dashed, mark=triangle*, mark options={fill=cyan, rotate=90}, mark repeat=2] table[x=SNR,y=MCP] {\leftData};
\addplot[magenta, dashed, mark=square*, mark options={fill=magenta}, mark repeat=2] table[x=SNR,y=Lp] {\leftData};
\addplot[black, mark=triangle*, mark options={fill=black}, mark repeat=2] table[x=SNR,y=TH] {\leftData};
\addplot[orange, mark=pentagon*, mark options={fill=orange}, mark repeat=2] table[x=SNR,y=L1divL2] {\leftData};

\nextgroupplot[
    xmin=30, xmax=72,
    xtick={30,39,48,57,66},
]
\addplot[blue, mark=*, mark options={fill=blue}, mark repeat=2] table[x=SNR,y=L1] {\rightData};
\addplot[red, mark=diamond*, mark options={fill=red}, mark repeat=2] table[x=SNR,y=TL1] {\rightData};
\addplot[green, mark=x, mark repeat=2] table[x=SNR,y=L1L2] {\rightData};
\addplot[magenta, mark=+, mark repeat=2] table[x=SNR,y=L0] {\rightData};
\addplot[cyan, dashed, mark=triangle*, mark options={fill=cyan, rotate=90}, mark repeat=2] table[x=SNR,y=MCP] {\rightData};
\addplot[magenta, dashed, mark=square*, mark options={fill=magenta}, mark repeat=2] table[x=SNR,y=Lp] {\rightData};
\addplot[black, mark=triangle*, mark options={fill=black}, mark repeat=2] table[x=SNR,y=TH] {\rightData};
\addplot[orange, mark=pentagon*, mark options={fill=orange}, mark repeat=2] table[x=SNR,y=L1divL2] {\rightData};

\end{groupplot}

\node[draw, font=\fontsize{6}{6}\selectfont, inner ysep=0.05cm] at ([xshift=-0.6cm,yshift=0.8cm]group c2r1.north) {
\begin{tikzpicture}
\matrix [column sep=0.05cm, ampersand replacement=\&] {
\node {\tikz{\draw[blue, line width=0.8pt] (0,0) -- (0.4,0); \node[blue] at (0.2,0) {\pgfuseplotmark{*}};}}; \& \node {$L_1$}; \&
\node {\tikz{\draw[red, line width=0.8pt] (0,0) -- (0.4,0); \node[red] at (0.2,0) {\pgfuseplotmark{diamond*}};}}; \& \node {$TL_1$}; \&
\node {\tikz{\draw[magenta, line width=0.8pt] (0,0) -- (0.4,0); \node[magenta] at (0.2,0) {\pgfuseplotmark{+}};}}; \& \node {$L_0$}; \&
\node {\tikz{\draw[cyan, dashed, line width=0.8pt] (0,0) -- (0.4,0); \node[cyan, rotate=90] at (0.2,0) {\pgfuseplotmark{triangle*}};}}; \& \node {$MCP$}; \&
\node {\tikz{\draw[magenta, dashed, line width=0.8pt] (0,0) -- (0.4,0); \node[magenta] at (0.15,0) {\pgfuseplotmark{square*}};}}; \& \node {$L_p$}; \&
\node {\tikz{\draw[green, line width=0.8pt] (0,0) -- (0.4,0); \node[green] at (0.2,0) {\pgfuseplotmark{x}};}}; \& \node {$L_1 - L_2$}; \&
\node {\tikz{\draw[orange, line width=0.8pt] (0,0) -- (0.4,0); \node[orange] at (0.2,0) {\pgfuseplotmark{pentagon*}};}}; \& \node {$L_1/L_2$};\&
\node {\tikz{\draw[black, line width=0.8pt] (0,0) -- (0.4,0); \node[black] at (0.2,0) {\pgfuseplotmark{triangle*}};}}; \& \node {$TH$}; \& \\
};
\end{tikzpicture}
};
\end{tikzpicture}
\caption{{Relative reconstruction error (RRE) versus signal-to-noise ratio (SNR) for sparse recovery using Gaussian (left) and oversampled discrete cosine transform (DCT) sensing matrices (right). Error values represent means over 50 independent trials.}}
\label{vary_SNR}
\end{figure}

\subsection{{Weighted truncated Huber model}}
{
Inspired by the sorted $L_1/L_2$ penalty, a recent state-of-the-art method in sparse signal recovery \cite{wang2024sorted}, we extend our $TH$ model to a weighted version, termed the Weighted Truncated Huber ($WTH$) model. 
This approach introduces a weight vector to assign varying penalties based on signal entry magnitudes, promoting sparsity while preserving significant components. 
The penalty function $\Phi_{\mu}(\vx)$ is modified to include weights $\vtau \in (0,1]^n$:
\[
\Phi_{\mu}(\vx, \vtau) = \sum_{i=1}^n \vtau_i \phi_{\mu}(\vx_i),
\]
where $\phi_{\mu}(\vx_i)$ is the Huber penalty function. The $WTH$ model for noisy sparse signal recovery is formulated as:
\begin{equation}
    {\vx} = \argmin_{\vx} \frac{\alpha}{2} \|A\vx - \vb\|_2^2 + \Phi_{\mu}(\vx, \vtau).
\end{equation}
To facilitate optimization, we define $\vtheta_i = \sqrt{\vtau_i}$ and introduce a surrogate function:
\[
\Phi_{\mu}(\vx, \vtau) = \min_{\vomega} \mathcal{Q}_{\mu,\vtheta}(\vx, \vomega) := \frac{1}{\mu^2} \|(\mathbf{1} - \vomega) \circ \vtheta \circ \vx\|_2^2 + \|\vomega\|_0,
\]
where $\vomega$ is an auxiliary variable promoting sparsity. The weights $\vtheta_i$ assign higher penalties to smaller entries, following \cite{Huang2015,wang2024sorted}:
\[
\vtheta_i =
\begin{cases} 
1, & \text{otherwise}, \\
e^{-r(t-i)/t}, & i \in \Gamma_{\vx,t},
\end{cases}
\]
where $r > 0$ controls the weight function’s curvature, and $\Gamma_{\vx,t}$ is the index set of the $t$ largest-magnitude entries of $\vx$.
}

{
Optimization is performed via  BCD method, alternating updates for $\vomega$ and $\vx$. Given $\vx$, the $\vomega$-subproblem minimizes $\mathcal{Q}_{\mu,\theta}(\vx, \vomega)$ component-wise. For $\vx_i = 0$, the minimum is achieved at ${\vomega}_i = 0$. For $\vx_i \neq 0$, the subproblem is:
\[
\min_{\vomega_i} \frac{ \vtheta_i^2 \vx_i^2}{\mu^2} (1 - \vomega_i)^2 + |\vomega_i|^0.
\]
This yields a hard thresholding rule:
\[
{\vomega}_i  =
\begin{cases} 
0, & |\vx_i| \leq \frac{\mu}{ \vtheta_i}, \\
1, & |\vx_i| > \frac{\mu}{ \vtheta_i},
\end{cases}
\]
corresponding to the thresholding operator $(\mathcal{H}_{\mu /\vtheta_i}(\vx))_i$, as defined in \eqref{porpoptw}. Given ${\vomega}$, the $\vx$-subproblem solves:
\[
\left( \alpha A^\top A +\frac{2}{\mu^2} \operatorname{diag}\big( (\mathbf{1} - {\vomega}) \circ \vtheta^{ 2} \big) \right) \vx = \alpha A^\top \vb.
\]
}


\begin{figure}
\centering
\begin{tikzpicture}
\pgfplotstableread{
m   L1     IRLS   L1mL2  L1vL2  SL1dL2 TH     WTH
30  0.8603 0.9008 0.8388 0.8512 0.9623 0.8603 0.8413
50  0.3864 0.1329 0.1657 0.3759 0.1700 0.2378 0.2843
70  0.0188 0.0028 0.0107 0.0072 0.0073 0.0042 0.0031
90  0.0111 0.0028 0.0084 0.0065 0.0068 0.0031 0.0025
110 0.0087 0.0024 0.0072 0.0063 0.0055 0.0025 0.0023
130 0.0077 0.0022 0.0071 0.0068 0.0055 0.0020 0.0018
150 0.0075 0.0026 0.0070 0.0066 0.0057 0.0018 0.0016
170 0.0078 0.0028 0.0073 0.0070 0.0060 0.0021 0.0018
190 0.0077 0.0032 0.0074 0.0070 0.0052 0.0019 0.0017
210 0.0079 0.0030 0.0076 0.0074 0.0050 0.0015 0.0015
}\leftData

\pgfplotstableread{
m   L1     IRLS   L1mL2  L1vL2  SL1dL2 TH     WTH
30  0.0260 0.0348 1.1405 0.5285 4.7798 0.0007 0.0297
50  0.0464 0.0149 0.8719 0.4431 4.5298 0.0104 0.0279
70  0.0775 0.0095 1.3430 0.5293 4.6269 0.0103 0.1102
90  0.1113 0.0117 1.4679 0.5361 4.5936 0.0131 0.0628
110 0.0980 0.0292 1.6360 0.5676 4.4157 0.0156 0.1897
130 0.1347 0.0315 1.8884 0.5891 4.5645 0.0182 0.5042
150 0.1710 0.0195 2.0633 0.6080 4.6027 0.0220 0.1552
170 0.2008 0.0193 2.2098 0.6303 4.6016 0.0256 0.3152
190 0.2227 0.0244 2.4125 0.6451 4.5920 0.0296 0.3630
210 0.2596 0.0285 2.6741 0.6817 4.4951 0.0325 0.3858
}\rightData

\begin{groupplot}[
    group style={
        group size=2 by 1,
        horizontal sep=1.6cm,
        vertical sep=1cm,
        group name=myplots,
    },
    width=6.2cm,
    height=4.5cm,
    xlabel={m},
    xlabel style={font=\fontsize{6}{6}\selectfont},
    xticklabel style={font=\fontsize{6}{6}\selectfont},
    yticklabel style={font=\fontsize{6}{6}\selectfont},
    grid=none,
    minor tick num=1,
]

\nextgroupplot[
    xmin=30, xmax=210,
    xtick={30,70,110,150,190},
    ymode=log,
    ylabel={RRE},
    ylabel style={font=\fontsize{6}{6}\selectfont},
]
\addplot[blue, mark=*, mark options={fill=blue}, mark repeat=2] table[x=m,y=L1] {\leftData};
\addplot[green, mark=x, mark repeat=2] table[x=m,y=L1mL2] {\leftData};
\addplot[magenta, mark=+, mark repeat=2] table[x=m,y=L1vL2] {\leftData};
\addplot[cyan, dashed, mark=triangle*, mark options={fill=cyan, rotate=90}, mark repeat=2] table[x=m,y=SL1dL2] {\leftData};
\addplot[black, mark=triangle*, mark options={fill=black}, mark repeat=2] table[x=m,y=TH] {\leftData};
\addplot[orange, mark=pentagon*, mark options={fill=orange}, mark repeat=2] table[x=m,y=WTH] {\leftData};

\nextgroupplot[
    xmin=30, xmax=210,
    xtick={30,70,110,150,190},
    ylabel={Time (s)},ymode=log,
    ylabel style={font=\fontsize{6}{6}\selectfont},
]
\addplot[blue, mark=*, mark options={fill=blue}, mark repeat=2] table[x=m,y=L1] {\rightData};
\addplot[green, mark=x, mark repeat=2] table[x=m,y=L1mL2] {\rightData};
\addplot[magenta, mark=+, mark repeat=2] table[x=m,y=L1vL2] {\rightData};
\addplot[cyan, dashed, mark=triangle*, mark options={fill=cyan, rotate=90}, mark repeat=2] table[x=m,y=SL1dL2] {\rightData};
\addplot[black, mark=triangle*, mark options={fill=black}, mark repeat=2] table[x=m,y=TH] {\rightData};
\addplot[orange, mark=pentagon*, mark options={fill=orange}, mark repeat=2] table[x=m,y=WTH] {\rightData};

\end{groupplot}

\node[draw, font=\fontsize{6}{6}\selectfont, inner ysep=0.05cm] at ([xshift=-0.6cm,yshift=0.8cm]group c2r1.north) {
\begin{tikzpicture}
\matrix [column sep=0.05cm, ampersand replacement=\&] {
\node {\tikz{\draw[blue, line width=0.8pt] (0,0) -- (0.4,0); \node[blue] at (0.2,0) {\pgfuseplotmark{*}};}}; \& \node {$L_1$}; \&
\node {\tikz{\draw[green, line width=0.8pt] (0,0) -- (0.4,0); \node[green] at (0.2,0) {\pgfuseplotmark{x}};}}; \& \node {$L_1 - L_2$}; \&
\node {\tikz{\draw[magenta, line width=0.8pt] (0,0) -- (0.4,0); \node[magenta] at (0.2,0) {\pgfuseplotmark{+}};}}; \& \node {$L_1 / L_2$}; \&
\node {\tikz{\draw[cyan, dashed, line width=0.8pt] (0,0) -- (0.4,0); \node[cyan, rotate=90] at (0.2,0) {\pgfuseplotmark{triangle*}};}}; \& \node {$S L_1 / L_2$}; \&
\node {\tikz{\draw[black, line width=0.8pt] (0,0) -- (0.4,0); \node[black] at (0.2,0) {\pgfuseplotmark{triangle*}};}}; \& \node {TH}; \&
\node {\tikz{\draw[orange, line width=0.8pt] (0,0) -- (0.4,0); \node[orange] at (0.2,0) {\pgfuseplotmark{pentagon*}};}}; \& \node {WTH}; \& \\
};
\end{tikzpicture}
};
\end{tikzpicture}
\caption{Relative reconstruction error (RRE) and computation time versus matrix row size (m) for sparse recovery using a Gaussian sensing matrix with column size 512 and added Gaussian noise (standard deviation 0.01). RRE and time values represent means over 10 independent trials.}
\label{vary_m}
\end{figure}

{
We evaluate the $WTH$ and $TH$ models on sparse signal recovery tasks, comparing them against state-of-the-art methods: $L_1$ \cite{glowinski1975approximation},   $L_1\mbox{-} L_2$ \cite{yin2015minimization}, $L_1/L_2$ \cite{tao2022minimization}, and sorted $L_1/L_2$ ($SL_1/L_2$) \cite{wang2024sorted}. Experiments use Gaussian sensing matrices $A_1 \in \mathbb{R}^{m \times 512}$, with $m$ varying from 30 to 210 in steps of 20, sparsity $s = 12$, and the observed data is corrupted by additive Gaussian noise with standard deviation $\sigma=0.01$. 
Signals are normalized as in \cite{wang2024sorted}. Performance is measured by RRE, averaged over 10 trials per $m$.
}

{
Figure~\ref{vary_m} shows RRE and computation time versus $m$. 
The proposed $WTH$ and $TH$ models outperform all baselines in 7 out of 10 cases ($m \geq 70$), achieving the lowest RRE values. Specifically, $WTH$ consistently yields the smallest RRE.
Both methods maintain stable performance across all cases, with $WTH$ showing slightly better results than $TH$.
In terms of efficiency, $TH$ demonstrates highly competitive computation times across all $m$. 
While $WTH$ incurs a modest computational overhead due to its weighting scheme, it remains significantly faster than $L_1$-$L_2$, $L_1/L_2$ and $SL_1/L_2$. 
These results highlight the superior accuracy and computational efficiency of $WTH$ and $TH$.
}

\subsection{Truncated Huber penalty in the gradient domain}

In addition, we further evaluate the versatility of the proposed $TH$ method by extending it to the gradient domain for denoising and smoothing tasks. This extension enables us to compare the performance of $TH$ with a range of denoising or smoothing techniques, including classical total variation ($TV$) \cite{rudin1992nonlinear}, $MC\mbox{-}TV$ \cite{selesnick2014convex}, $GME\mbox{-}TV$ \cite{selesnick2020non}, bilateral filtering ($BF$) \cite{tomasi1998bilateral}, weighted least squares filtering ($WLS$) \cite{farbman2008edge}, rolling guidance filter ($RGF$) \cite{zhang2014rolling}, {iterative least squares ($ILS$) \cite{10.1145/3388887}, contrastive semantic-guided image smoothing network ($CSGIS\mbox{-}Net$) \cite{https://doi.org/10.1111/cgf.14681}, pyramid texture filtering ($PTF$) \cite{10.1145/3592120}, edge-preserving smoothing framework ($GSF$) \cite{9490302},  Easy2Hard network ($EH\mbox{-}Net$) \cite{9451539} and guided image filtering ($GF$) \cite{6319316}.} 
{The parameter was chosen to optimize the numerical and visual performance across experiments.}

\begin{figure}[htb]
\centering
\begin{tabular}{ccc}
\begin{minipage}{0.28\textwidth}
  \centering
  \includegraphics[width=\textwidth]{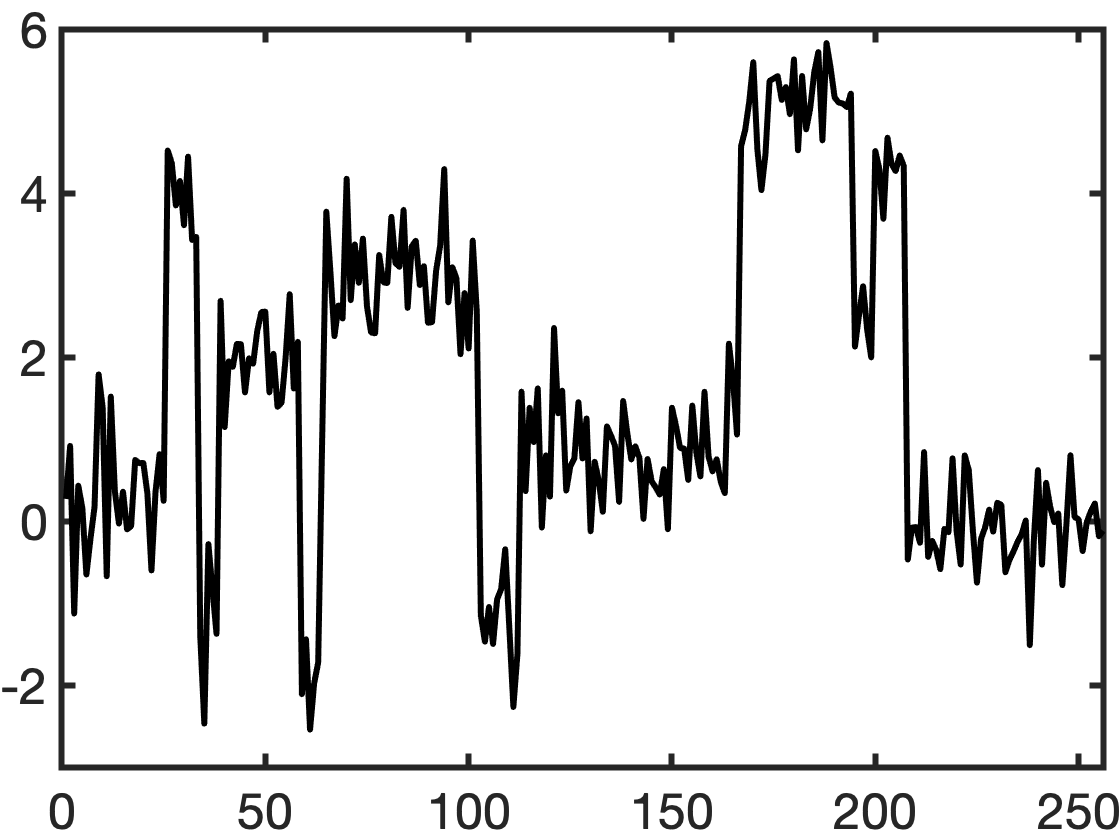}\\
  (a) Noisy signal
\end{minipage}
&
\begin{minipage}{0.28\textwidth}
  \centering
  \includegraphics[width=\textwidth]{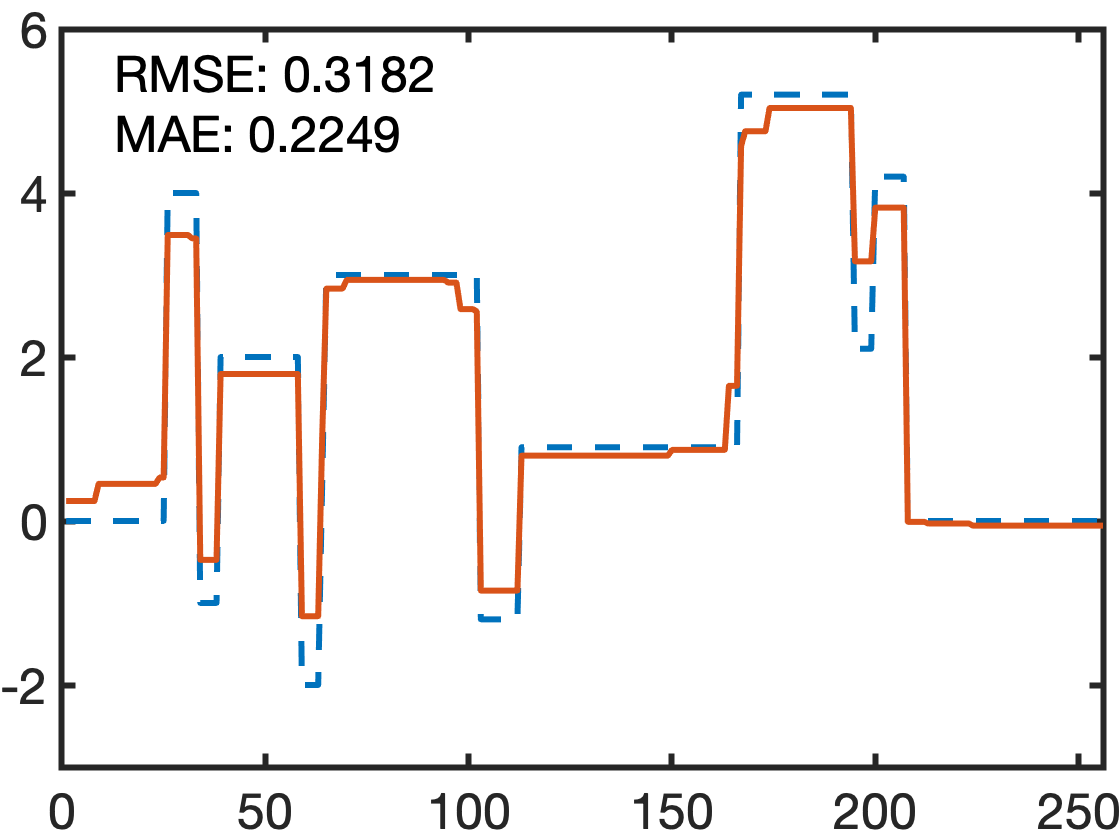}\\
  (b) $TV$
\end{minipage}
&
\begin{minipage}{0.28\textwidth}
  \centering
  \includegraphics[width=\textwidth]{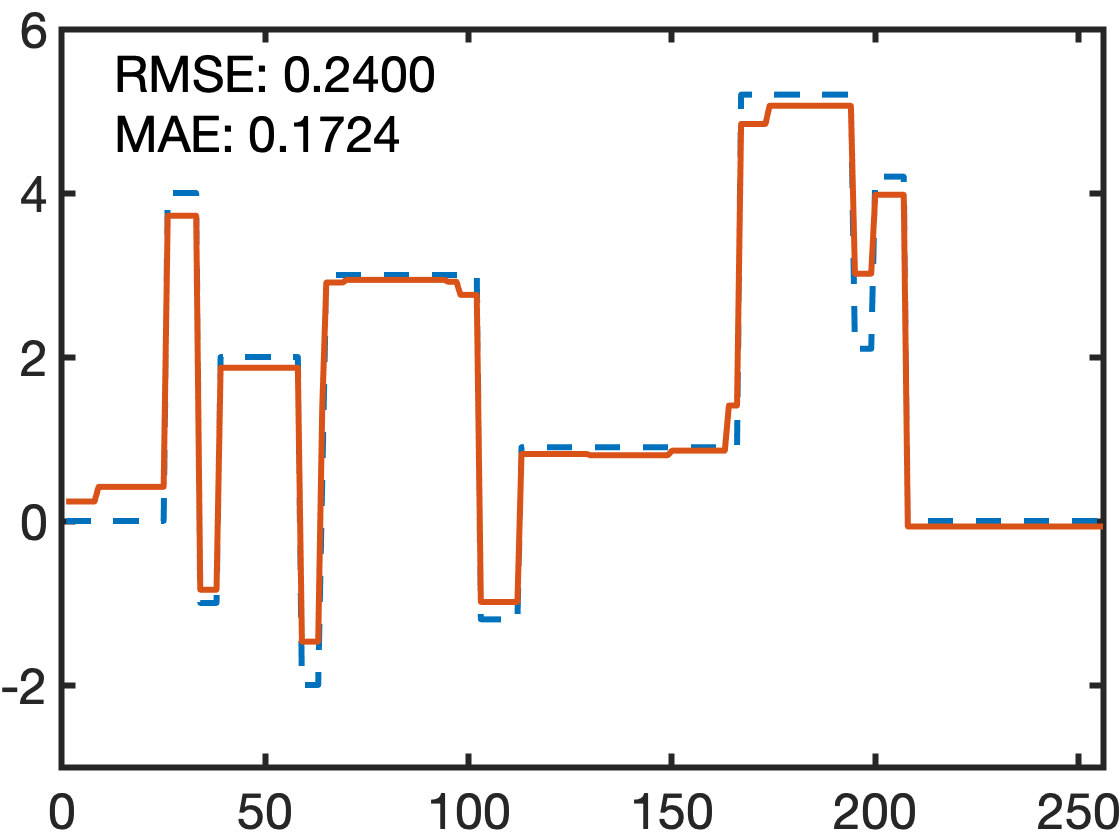}\\
  (c) $MC\mbox{-}TV$
\end{minipage}
\\[6pt]
\begin{minipage}{0.28\textwidth}
  \centering
  \includegraphics[width=\textwidth]{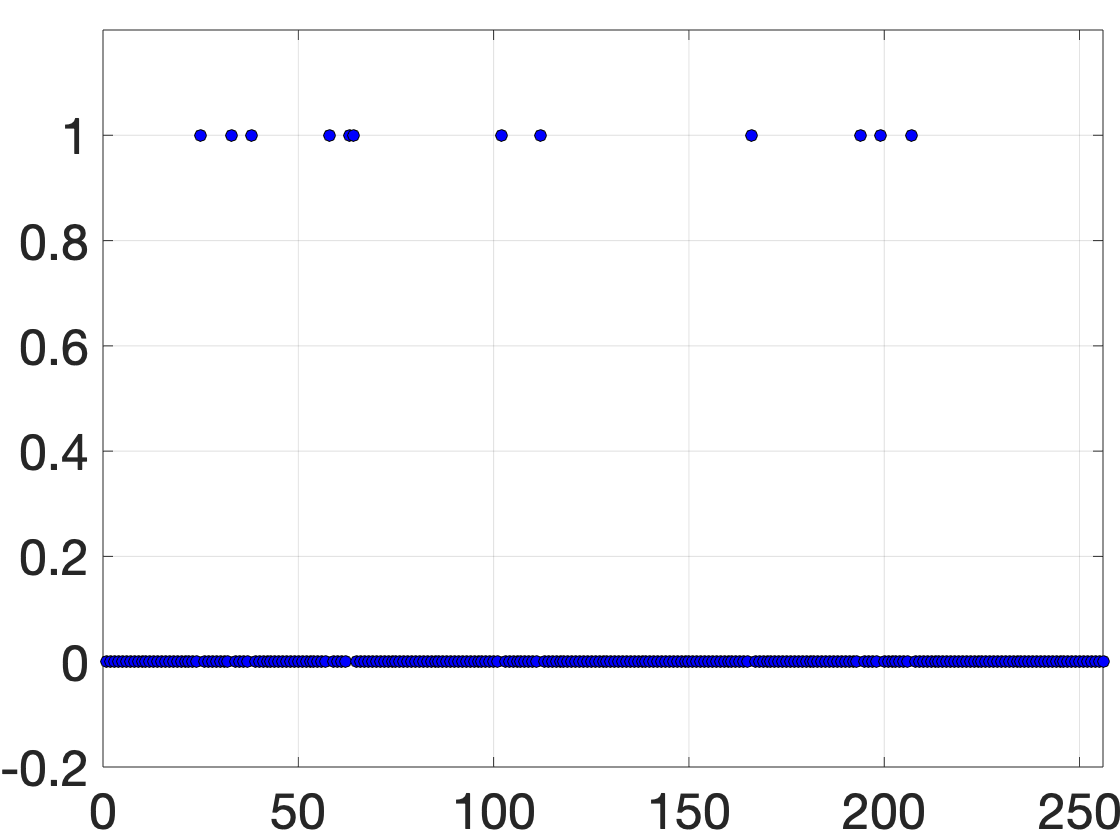}\\
  (d) Recovered $\vomega$
\end{minipage}
&
\begin{minipage}{0.28\textwidth}
  \centering
  \includegraphics[width=\textwidth]{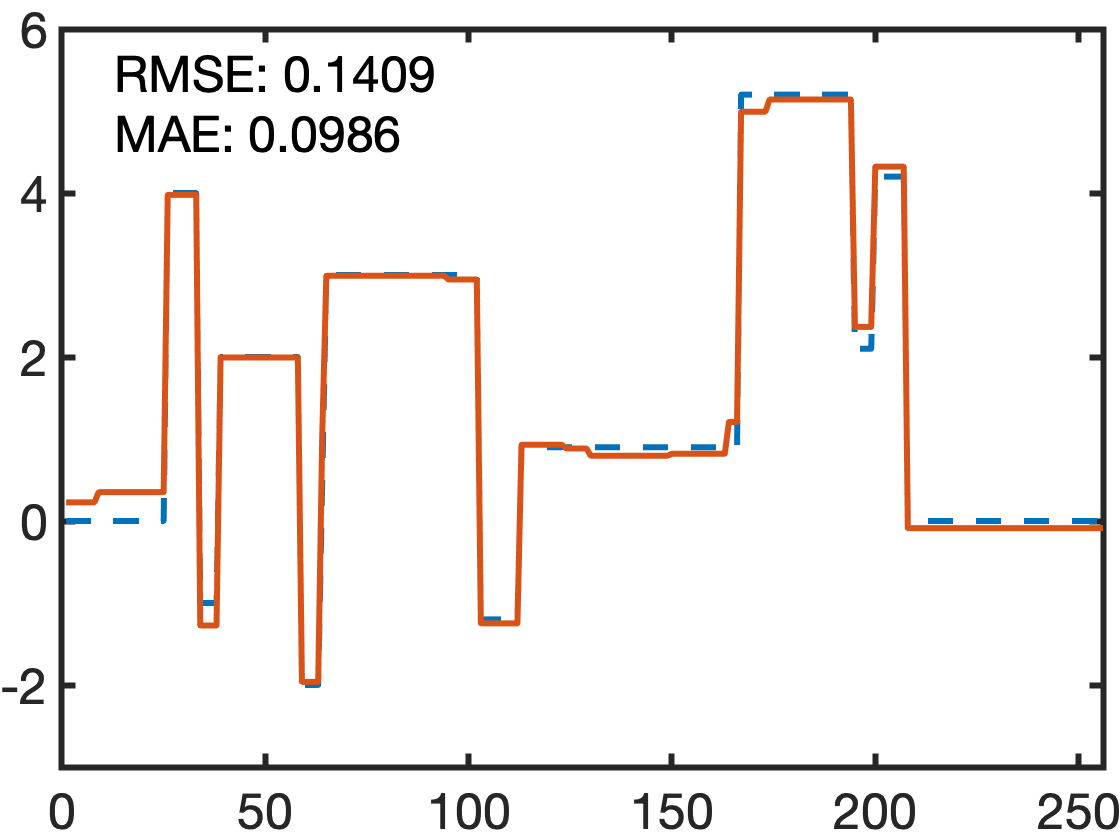}\\
  (e) $GME\mbox{-}TV$
\end{minipage}
&
\begin{minipage}{0.28\textwidth}
  \centering
  \includegraphics[width=\textwidth]{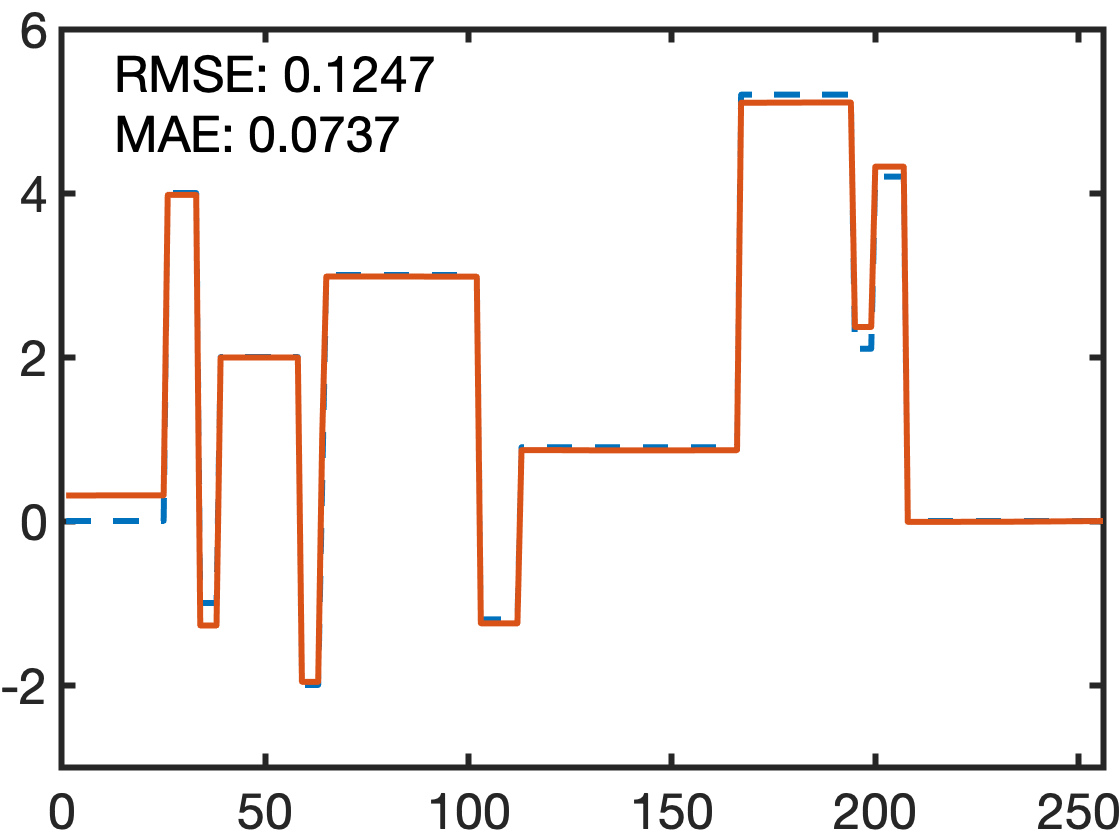}\\
  (f) $TH$
\end{minipage}
\end{tabular}
\caption{
(a) Noisy input signal with discontinuities; 
    (d) Recovered $w$ component (facilitating automatic jump detection); 
    (b)-(c) and (e)-(f) Reconstruction results using classical $TV$, $MC\mbox-TV$, $GME\mbox-TV$, and $TH$, respectively. 
    Dashed lines denote the ground truth signal, while solid lines represent the recovered signals.}
\label{fig:comparison}
\end{figure}

Given an observed signal $\vb\in\mathbb{R}^{n}$ contaminated by noise, i.e., 
$$\vb = \vx + \vn,$$ where $\vx\in\mathbb{R}^{n}$ is the desired signal and $\vn\in\mathbb{R}^{n}$ represents noise or fine-scale details to be removed.

By applying a surrogate function in the gradient domain (see Lemma~\ref{upperbound_lema2_rev}), we propose the following optimization problem:
\begin{equation}\label{gredient_mdl}
(\wvx,\widehat{\vomega})\in\arg\min_{\vx,\vomega}\ \frac{\alpha}{2}\|\vx-\vb\|_2^2 + \frac{1}{\mu^2}\|(\mathbf{1}-\vomega)\circ \nabla \vx\|_2^2 + \|\vomega\|_0,
\end{equation}
where $\nabla$ denotes the discrete gradient operator and $\alpha>0$ is a regularization parameter. Here, the auxiliary variable $\vomega\in\rr^n$ plays a key role in adaptively distinguishing between edges and non-edge regions without requiring any prior knowledge of their number or locations.
By employing the BCD scheme, the iterative updates are derived as:
\[
\left\{
\begin{aligned}
&\wvx^{j+1} = \Bigl(\operatorname{I} + \lambda\,\nabla^{\top}\operatorname{diag}(\mathbf{1}-\widehat{\vomega}^{j})^2 \nabla\Bigr)^{-1}\vb,\\[1mm]
&\widehat{\vomega}^{j+1} = \cH_{\mu}(\nabla \wvx^{j+1}),
\end{aligned}
\right.
\]
where $\lambda=\tfrac{2}{\alpha\mu^2}$, and $\cH_{\mu}$ is defined in \eqref{porpoptw}.

Building upon the comprehensive experimental framework introduced above, we compare our method with several state-of-the-art denoising and smoothing techniques. The following experiments on both one-dimensional (1D) and two-dimensional (2D) signals evaluate the ability of $TH$ to preserve sharp edges while effectively suppressing noise and fine-scale details.

For 1D experiments, the true signal is a piecewise constant signal generated by the \texttt{MakeSignal} function from the Wavelab software library. The noisy observation (see Figure~\ref{fig:comparison}(a)) is obtained by contaminating the true signal with additive white Gaussian noise ($\sigma=0.5$).  
As shown in Figure~\ref{fig:comparison}, our $TH$ method not only achieves the best performance in terms of root-mean-square error (RMSE) and mean absolute error (MAE) but also uniquely preserves sharp edges. This advantage is attributed to the auxiliary variable $\vomega$ (see Figure~\ref{fig:comparison}(d), introduced in Lemma~\ref{upperbound_lema2_rev}), which facilitates accurate edge detection without any prior knowledge regarding the number or locations of discontinuities. In contrast,  classical $TV$ tends to underestimate discontinuities and does not yield a piecewise constant reconstruction, whereas both $MC\mbox{-}TV$ and $GME\mbox{-}TV$ produce reconstructions closer to the ground truth, yet still inferior to the performance of the $TH$.

\newcommand{\addimwoman}[1]{
  \begin{tikzpicture}[spy using outlines={rectangle, magnification=4, size=37.5pt, every spy on node/.append style={very thick}}]
    \node[anchor=south west, inner sep=0] at (0,0)
      {\includegraphics[height=76pt]{#1}};
    \spy [yellow] on (11pt,53pt) in node [right] at (0pt,-18pt);
    \spy [blue] on (25pt,28pt) in node [right] at (38pt,-18pt);
    \spy [red] on (40pt,35pt) in node [right] at (77pt,-18pt);
  \end{tikzpicture}
}

\begin{figure}[htb]
  \centering
  \setlength{\tabcolsep}{0.0em}
  \begin{tabular}{ccc}
    \addimwoman{figures1125/woman_yt.PNG} & 
    \addimwoman{figures1125/woman_bf.PNG} & 
    \addimwoman{figures1125/woman_GF.PNG} \\
    (a) Original image & (b) BF \cite{tomasi1998bilateral} & (c) GF \cite{6319316} \\
    \addimwoman{figures1125/woman_TV.PNG} & 
    \addimwoman{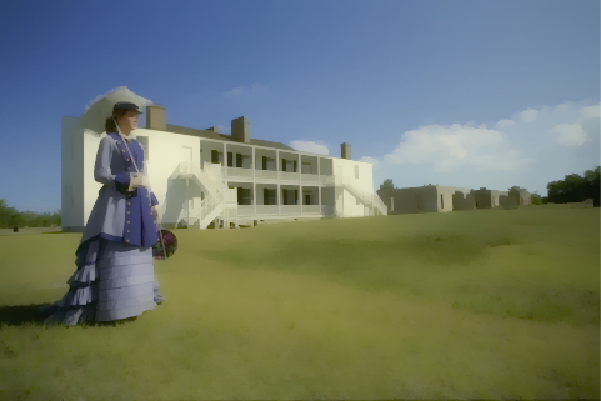} & 
    \addimwoman{figures1125/woman_GS.PNG} \\
    (d) TV \cite{rudin1992nonlinear} & (e) WLS \cite{farbman2008edge} & (f) GSF \cite{9490302} \\
    \addimwoman{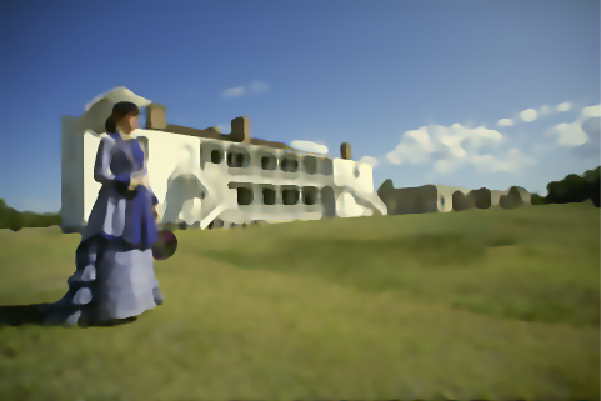} &
    \addimwoman{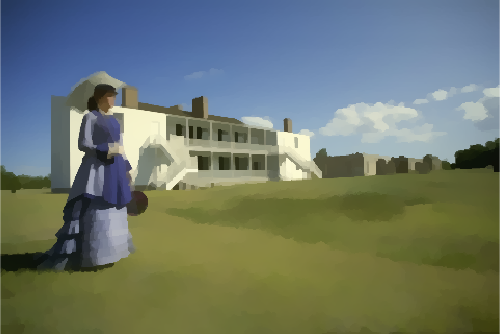} &
    \addimwoman{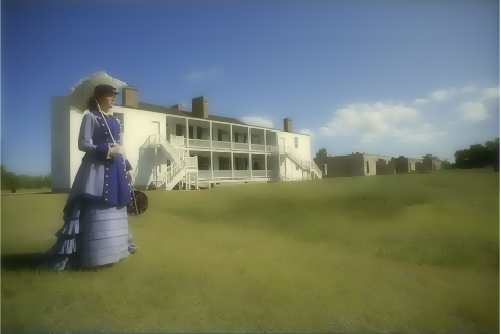} \\
    (g) RGF \cite{zhang2014rolling} & (h) PTF \cite{10.1145/3592120} & (i) ILS \cite{10.1145/3388887}  \\
    \addimwoman{figures1125/woman_CSGIS.JPG} & 
    \addimwoman{figures1125/woman_EH.JPG} &
    \addimwoman{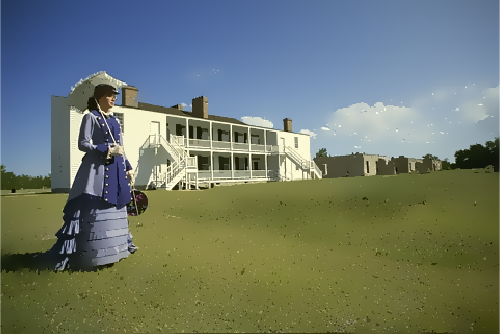}  \\
    (j) CSGIS-Net \cite{https://doi.org/10.1111/cgf.14681} & (k) EH-Net \cite{9451539}  & (l) TH  \\
  \end{tabular}
  \caption{{Comparison of smoothing performance across different methods. }}
  \label{women}
\end{figure}

For 2D experiments, we present the smoothing results on a natural image in Figure~\ref{women}. Three regions of interest are highlighted:
{
\begin{enumerate}[1),leftmargin=*]
    \item building-sky boundary (yellow): the proposed $TH$ method effectively preserves sharp edges, in contrast to the $RGF$, $ILS$, $CSGIS\mbox{-}Net$, and $EH\mbox{-}Net$, which significantly blur the boundary and content. Other methods, such as $BF$, $GSF$, $PTF$, and $TV$, also exhibit edge-preserving capabilities, whereas $GF$ (Figure \ref{women}(c)) and $WLS$ (Figure \ref{women}(e)) maintain edges but with less clarity. 
    \item clothing details (blue): the $TH$ method smooths wrinkles while maintaining distinct layer boundaries. Although $BF$, $GF$, $ILS$, $CSGIS\mbox{-}Net$, and $EH\mbox{-}Net$ preserve edges, they do not smooth fine details as effectively. Meanwhile, $WLS$ (Figure \ref{women}(e)) demonstrates weak edge preservation with insufficient wrinkle smoothing. The $RGF$ (Figure \ref{women}(g)) excessively blurs details, erasing buttons and skirt edges, while $TV$, $GSF$, and $PTF$ methods oversmooth the skirt, distorting a single layer into multiple parts and losing button details (Figures \ref{women}(d), \ref{women}(f) and \ref{women}(h)).
    \item staircase structure (red): the $TH$ model not only retains the staircase structure but also smooths the surrounding textures. Other methods either oversmooth or fail to capture these details. The $GF$ and $CSGIS\mbox{-}Net$ also perform well in this region, $BF$ fails to adequately smooth the lawn texture, and other methods do not maintain the staircase details as effectively.
\end{enumerate}
}

\section{Conclusions}\label{sec:concl}

A truncated Huber penalty framework for sparse signal recovery has been considered in this paper. The proposed penalty function combines the unbiasedness of non-convex metrics with the computational efficiency of convex methods by penalizing signal entries based on their magnitudes.
Theoretically, we have proven that any $s$-sparse solution remains local optima under TH, differentiability at convergence points, and finite-step convergence of the proposed block coordinate descent algorithm under spark conditions. 
These theoretical insights are complemented by comprehensive numerical validations.
In noise-free settings, the method achieves fast recovery, outperforming most methods nearly 10 times. For noisy observations, it maintains the lowest RRE the entire range of SNR levels.
{We proposed a weighted extension of the $TH$ penalty ($WTH$), inspired by sorted penalties. The $WTH$ model adaptively assigns penalties to signal components and consistently improves upon the original $TH$ and other baseline methods across various noise levels.}
Extensions to the gradient highlight the framework’s versatility, where edge preservation and noise suppression outperform all the compared methods. The continuation strategy for $\mu$ ensures numerical stability, though regularization parameter tuning remains a practical consideration. 
{Future work will explore adaptive $\mu$-selection schemes, deeper analysis of the weighted model’s theoretical guarantees, theoretical properties of its gradient domain extension, integration with deep learning approaches, and applications to other image processing tasks.}

\bibliographystyle{siamplain}
\bibliography{reference}
\end{document}

%% file: symboldef.tex
\def\va{{\boldsymbol a}}
\def\vb{{\boldsymbol b}}
\def\vc{{\boldsymbol c}}

\def\ve{{\boldsymbol e}}

\def\vn{{\boldsymbol n}}

\def\vv{{\boldsymbol v}}

\def\vx{{\boldsymbol x}}
\def\vy{{\boldsymbol y}}
\def\vz{{\boldsymbol z}}

\def\wvx{\widehat{{\boldsymbol x}}}
\def\wvy{\widehat{{\boldsymbol y}}}

\def\vomega{\boldsymbol{\omega}}

\def\vtau{\boldsymbol{\tau}}
\def\vtheta{\boldsymbol{\theta}}

\def\cG{\mathcal{G}}
\def\cH{\mathcal{H}}
\def\cI{\mathcal{I}}
\def\cJ{\mathcal{J}}

\def\cL{\mathcal{L}}

\def\cQ{\mathcal{Q}}

\def\cS{\mathcal{S}}

\def\argmin{\mathop{\rm argmin}}

\newcommand\norm[1]{\left\Vert#1\right\Vert}
\newcommand\abs[1]{\left|#1\right|}

\newcommand\LS[1]{\left\|#1\right\|_2^2}

\definecolor{ecolor}{RGB}{1,126,218}%
\definecolor{winered}{rgb}{0.5,0,0}
\definecolor{epubblue}{RGB}{1,126,218}
\definecolor{egreen}{RGB}{0,120,2}
\definecolor{ecyan}{RGB}{0,175,152}
\definecolor{eblue}{RGB}{20,50,104}
\definecolor{sakura}{RGB}{255,183,197}
\definecolor{brown}{RGB}{109,62,18}
\definecolor{lightgrey}{rgb}{0.9,0.9,0.9}
\definecolor{frenchplum}{RGB}{190,20,83}
\definecolor{myfoot}{rgb}{0.5,0.2,0.5}
\definecolor{darkblue}{rgb}{0.1,0,0.85}
\definecolor{studentbrown}{RGB}{124,71,50}

\definecolor{bronze}{RGB}{205,127,50}
\definecolor{grotto}{RGB}{14,130,212}
\definecolor{purplea}{RGB}{179,137,240}
\definecolor{greenb}{RGB}{111,186,193}
\definecolor{greena}{RGB}{143,169,73}
\definecolor{redb}{RGB}{217,129,118}
\definecolor{yellowa}{RGB}{230,155,96}
\definecolor{reda}{RGB}{221,109,92}
\definecolor{purpleb}{RGB}{169,140,186}
\definecolor{mygreen}{RGB}{122,187,114}
\definecolor{darkblue}{RGB}{100,130,178}
\definecolor{halfblue}{RGB}{168,194,215}
\definecolor{yellowb}{RGB}{249,236,186}
\definecolor{orangeh}{RGB}{237,192,145}
\definecolor{MatchaGreen}{HTML}{73C088}		
\definecolor{MelancholyBlue}{HTML}{9EAABA}	
\definecolor{PureBlue}{HTML}{80A3D0}
\definecolor{SchembriumYellow}{HTML}{fbd26a}	
\definecolor{PeachRed}{HTML}{EA868F}
\definecolor{fadedgold}{HTML}{D9CBB0}		
\definecolor{saturatedgold}{HTML}{F0E0C2}	
\definecolor{elegantblue}{HTML}{C4CCD7}		
\definecolor{ivory}{HTML}{F1ECE6}			
\definecolor{gloomypruple}{HTML}{CCC1D2}	

\definecolor{tealgreen}{RGB}{32, 178, 170}
\definecolor{brightorange}{RGB}{255, 140, 0}
\definecolor{crimsonred}{RGB}{220, 20, 60}
\definecolor{violetpurple}{RGB}{138, 43, 226}
\definecolor{dodgerblue}{RGB}{30, 144, 255}
\definecolor{darkgolden}{RGB}{229, 193, 0} 
\definecolor{deepgreen}{RGB}{0, 116, 0}
\definecolor{copperbrown}{RGB}{198, 122, 55}
\definecolor{emerald}{RGB}{58,191,153} 